\documentclass[arxiv,reqno,twoside,a4paper,12pt]{amsart}

\usepackage{tikz}
\usepackage{amsmath, verbatim}
\usepackage{amssymb,amsfonts,mathrsfs}
\usepackage[colorlinks=true,linkcolor=blue,citecolor=blue]{hyperref}
\usepackage[driver=pdftex,heightrounded=true,centering]{geometry}
\usepackage{longtable}
\usepackage{tabularx}
\usepackage{multirow}
\usepackage{caption}

\usepackage[scaled]{helvet} 
\usepackage{courier} 
\usepackage[mathbf]{euler}

\theoremstyle{plain}
\newtheorem{theorem}{Theorem}[section]
\newtheorem{thm}{Theorem}[section]
\newtheorem{prop}[thm]{Proposition}
\newtheorem{cor}[thm]{Corollary}
\newtheorem{defn}[thm]{Definition}
\newtheorem{lem}[thm]{Lemma}

\newtheorem{remark}[thm]{Remark}

\theoremstyle{definition}

\theoremstyle{remark}

         \newtheorem{ex}[theorem]{Example}

\theoremstyle{plain}

\newcommand{\R}{\mathbb{R}}

\newcommand{\N}{\mathbb{N}}

\newcommand{\CP}{\mathbb{CP}}
\newcommand{\HP}{\mathbb{HP}}

\newcommand{\dimn}{\mathrm{dim}}

\newcommand{\supp}{\mathrm{supp}}

\newcommand{\ric}{\mathrm{Ric}}
\newcommand{\trace}{\mathrm{tr}}

\newcommand{\dv}{\text{ }dV}

\DeclareMathOperator{\cU}{\mathscr{U}}
\DeclareMathOperator{\cV}{\mathscr{V}}

\DeclareMathOperator{\cF}{\mathscr{F}}
\DeclareMathOperator{\cM}{\mathscr{M}}
\DeclareMathOperator{\cH}{\mathscr{H}}
\DeclareMathOperator{\A}{\alpha}
\DeclareMathOperator{\w}{\omega}
\DeclareMathOperator{\V}{\mathcal{V}}
\DeclareMathOperator{\cC}{\mathscr{C}}
\DeclareMathOperator{\dom}{\mathscr{D}}

\DeclareMathOperator{\ho}{\mathcal{C}^{\A}_{\textup{ie}}}

\DeclareMathOperator{\hok}{\mathcal{C}^{k, \A}_{\textup{ie}}}

\numberwithin{equation}{section}

\definecolor{qqwuqq}{rgb}{0,0,0}

\begin{document}

\date{\today}

\title[Stability of Ricci de Turck flow on Singular Spaces]
{Stability of Ricci de Turck flow on Singular Spaces}

\author{Klaus Kr\"oncke}
\address{University Hamburg, Germany} 
\email{klaus.kroencke@uni-hamburg.de} 

\author{Boris Vertman} 
\address{Universit\"at Oldenburg, Germany} 
\email{boris.vertman@uni-oldenburg.de}

\thanks{Partial support by DFG Priority Programme "Geometry at Infinity"}


\begin{abstract}
In this paper we establish stability of the Ricci de Turck flow near Ricci-flat metrics
with isolated conical singularities. More precisely, we construct a Ricci de Turck flow
which starts sufficiently close to a Ricci-flat metric with isolated conical singularities
and converges to a singular Ricci-flat metric under an assumption of 
integrability, linear and tangential stability. We provide a characterization of conical singularities
satisfying tangential stability and discuss examples where the integrability condition
is satisfied. 
\end{abstract}

\maketitle

\tableofcontents

\section{Introduction and statement of the main results}

Geometric flows, among them most notably the Ricci flow, provide a powerful tool 
to attack classification problems in differential geometry and construct Riemannian 
metrics with prescribed curvature conditions. The interest in this research area only grew
since the Ricci flow was used decisively in the Perelman's proof of Thurston's geometrization 
and the Poincare conjectures. \medskip

The present work is a continuation of a research program on the 
Ricci flow in the setting of singular spaces. The two-dimensional Ricci flow
reduces to a scalar equation and has been studied on surfaces with conical 
singularities by Mazzeo, Rubinstein and Sesum in \cite{MRS} and Yin \cite{Yin:RFO}. 
The Ricci flow in two dimensions is equivalent to the Yamabe flow,
which has been studied in general dimension on spaces with edge singularities
by Bahuaud and the second named author in \cite{BV} and \cite{BV2}. 
\medskip 

In the setting of K\"ahler manifolds, K\"ahler-Ricci flow reduces to a scalar Monge 
Ampere equation and has been studied in case of edge singularities 
in connection to the recent resolution of the Calabi-Yau conjecture on Fano 
manifolds by Donaldson \cite{Donaldson} and Tian \cite{Tian}, by
Jeffres, Mazzeo and Rubinstein \cite{JMR}, Chen and Wang \cite{Wang1}, 
Wang \cite{Wang2}. \medskip

We should point out that in the singular setting, Ricci flow loses its uniqueness
and need not preserve the given singularity structure. In fact, Giesen and 
Topping \cite{Topping, Topping2} constructed a solution to the Ricci flow on 
surfaces with singularities, which becomes instantaneously complete. Alternatively,
Simon \cite{MS} constructed Ricci flow in dimension two and three that smoothens 
out the singularity. \medskip

In the present discussion, which can be viewed as a continuation of the 
recent work by the second named author in \cite{Ver-Ricci}, we consider 
Ricci de Turck flow preserving isolated conical singularities and establish a
stability result near Ricci-flat metrics. The crucial difficulty in our setting is 
in addition to the singularity of the underlying space the tensorial nature
of the flow, in contrast to the two-dimensional or the K\"ahler setting.
\medskip

We now proceed as follows. We first recall geometric aspects
of isolated conical singularities and define H\"older spaces adapted to the 
singular geometry and mapping properties of the heat kernel as in \cite{Ver-Ricci}. 
We then conclude the introduction with statement of the main results.

\subsection{Isolated conical singularities}

\begin{defn}\label{cone-metric}
Consider a compact smooth manifold $\overline{M}$ with boundary 
$\partial M = F$ and open interior denoted by $M$. 
Let $\overline{\cC(F)}$ be a tubular neighborhood of the boundary, 
with open interior $\cC(F) = (0,1)_x \times F$, where $x$ is a 
defining function of the boundary. Consider a smooth Riemannian metric
$g_F$ on the boundary $F$. An incomplete Riemannian metric $g$ on $M$ 
with an isolated conical singularity is then defined to be smooth away from the
boundary and 
\begin{equation*}
g \restriction \cC(F) = dx^2 + x^2 g_F + h,
\end{equation*}
where the higher order term $h$ is smooth on $\overline{\cC(F)}$
with the asymptotics $|h(x)|_{\overline{g}} = O(x)$ as $x \to 0$
for $\overline{g} = dx^2 + x^2 g_F$. 
\end{defn}

We call $(M,g)$ a compact space with an isolated conical singularity, 
or a \emph{conical manifold}. The definition naturally extends to conical manifolds 
with finitely many isolated conical singularities.
Since the analytic arguments are local in nature, we may assume without loss
of generality that $M$ has a single conical singularity only.
\medskip

In the present discussion we study Ricci-flat spaces $(M,g)$ with 
isolated conical singularities. There are various examples for such spaces. 
Consider a Ricci-flat smooth compact manifold $X$, e.g. a Calabi-Yau manifold or flat torus, 
with a discrete group $G$ acting by isometries, which is not necessarily 
acting strictly discontinuous and admits finitely many fixed points. 
The interior of its quotient $X/G$ defines a compact manifold, an orbifold, with isolated
conical singularities. There exist also examples of compact Ricci-flat manifolds
with non-orbifold isolated conical singularities, constructed by Hein and Sun 
\cite{HeSu}.

\subsection{Geometry of conical manifolds} \medskip

In this subsection we recall elements of b-calculus by Melrose \cite{Mel:TAP, Mel2}.
We choose local coordinates $(x,z)$ on the conical neighborhood 
$\cC(F)$, where $x$ is the defining function of the boundary, $n= \dim F$ and 
$(z)=(z_1,\ldots, z_n)$ are local coordinates on $F$. We
consider the Lie algebra of b-vector fields $\V_b$, which by definition are smooth 
on the closure $\overline{M}$ and tangent to the boundary $\partial M = F$. 
In local coordinates $(x,z)$, b-vector fields $\V_b$ are locally generated by 
\[
\left\{x\frac{\partial}{\partial x},  
\partial_z = \left( \frac{\partial}{\partial z_1},\dots, \frac{\partial}{\partial z_n} \right)\right\},
\]
with coefficients being smooth on $\overline{M}$. 
The b-vector fields form a spanning set of section for the 
b-tangent bundle ${}^bTM$, i.e. $\mathcal{V}_b=C^\infty(\overline{M},{}^bTM)$. 
The b-cotangent bundle ${}^bT^*M$ is generated locally by the following one-forms
\begin{align}\label{triv}
\left\{\frac{dx}{x}, dz_1,\dots,dz_n\right\}.
\end{align}
These differential forms are singular in the usual sense, but smooth as sections of 
the b-cotangent bundle ${}^bT^*M$. We extend $x:\overline{\cC(F)} \to [0,1]$ 
smoothly to $\overline{M}$, nowhere vanishing on $M$, and define the incomplete 
b-tangent space ${}^{ib}TM$ by the requirement $xC^\infty(\overline{M},{}^{ib}TM) := 
C^\infty(\overline{M},{}^{b}TM)$. The dual incomplete b-cotangent bundle 
${}^{ib}T^*M$ is related to its complete counterpart by 
\begin{align}
C^\infty(\overline{M},{}^{ib}T^*M) = 
x C^\infty(\overline{M},{}^{b}T^*M), 
\end{align}
with the spanning sections given locally over $\cC(F)$ by 
\begin{align}\label{triv2}
\left\{dx, x dz_1,\dots, x dz_n\right\}.
\end{align}
With respect to the notation we just introduced, the conical 
metric $g$ in Definition \ref{cone-metric} is a smooth section of 
the vector bundle of the symmetric $2$-tensors of the incomplete
b-cotangent bundle ${}^{ib}T^*M$, i.e. $g \in C^\infty (\textup{Sym}^2({}^{ib}T^*M))$.

\subsection{Statement of the main results} \medskip

Our main result establishes long time existence and convergence of the
Ricci de Turck flow for sufficiently small perturbations of Ricci-flat metrics
with isolated conical singularities, assuming tangential stability and some integrability
conditions. More precisely we consider a compact Ricci-flat manifold $(M,h_0)$ with an isolated 
conical singularity and $g_0$ a sufficiently small perturbation of $h_0$, not necessarily
Ricci-flat. We study the Ricci de Turck flow with $h_0$ as the reference metric, 
and $g_0$ as the initial metric 
\begin{equation}
\partial_t g(t) = -2 \textup{Ric} (g(t)) + \mathcal{L}_{W(t)} g(t), \quad g(0) = g_0,
\end{equation}
where $W(t)$ is the de Turck vector field defined in terms of the Christoffel symbols for 
the metrics $g(t)$ and $h_0$ 
\begin{equation}
W(t)^k = g(t)^{ij} \left(\Gamma^k_{ij}(g(t)) - \Gamma^k_{ij}(h_0)\right). 
\end{equation}

Our main result is as follows. 

\begin{thm}
Consider a compact Ricci-flat manifold $(M,h_0)$ with isolated conical singularities. 
Assume that $(M,h_0)$ satisfies the following three additional assumptions
\begin{enumerate}
\item[(i)] $(M,h_0)$ is tangentially stable in the sense of Definition \ref{tangential-stability-def}, 
\item[(ii)] $(M,h_0)$ is linearly stable in the sense of Definition \ref{linear-stability}, 
\item[(iii)] $(M,h_0)$ is integrable in the sense of Definition \ref{integrability}.
\end{enumerate}
If $h_0$ is not strictly tangentially stable, 
we assume in addition that the singularities are orbifold singularities. 
Then for sufficiently small perturbations $g_0$ of $h_0$, there exists a Ricci-de-Turck flow, with 
a change of reference metric at discrete times, starting at $g_0$ and converging to a Ricci-flat 
metric $h^*$ with isolated conical singularities at infinite time.
\end{thm}

We point out that linear stability and integrability are also imposed in the classical case to get stability
of the Ricci flow. The additional feature of isolated conical singularities is the assumption of tangential stability 
and the fact that we change the reference metric at discrete times in order to converge
to a Ricci-flat metric. \medskip

A crucial part of our paper is devoted to a detailed discussion of the tangential 
stability and integrability assumptions. For the former assumption we prove the following 
general characterization.

\begin{thm}
	Let $(F,g_F)$, $n\geq 3$ be a compact Einstein manifold with constant $(n-1)$. 
	We write $\Delta_E$ for its Einstein operator, and denote the Laplace Beltrami 
	operator by $\Delta$. Then $(F,g_F)$ is tangentially stable if and only if $\mathrm{Spec}(\Delta_E|_{TT})\geq0$ and $\mathrm{Spec}(\Delta)\setminus \left\{0\right\}\cap (n,2(n+1))=\varnothing$. Similarly, $(F,g_F)$ is strictly tangentially stable if and only if $\mathrm{Spec}(\Delta_E|_{TT})>0$ and $\mathrm{Spec}(\Delta)\setminus \left\{0\right\}\cap [n,2(n+1)]=\varnothing$. 
\end{thm}

We explain that any spherical space form is tangentially stable. 
However, the spaces $\mathbb{S}^n$ and $\R\mathbb{P}^n$ are not strictly tangentially stable. 
This property may also hold for other spherical space forms. We also provide a
detailed list of tangentially stable Einstein manifolds that are symmetric spaces. 

\begin{thm}
	Let $(F^n,g_F)$, $n\geq 2$ be a closed Einstein manifold with constant $(n-1)$, 
	which is a symmetric space of compact type. If it is a simple Lie group $G$, 
	it is strictly tangentially stable if $G$ is one of the following spaces:
	\begin{align}
	\mathrm{Spin}(p)\text{ }(p\geq 6,p\neq 7),\qquad \mathrm{E}_6,
	\qquad\mathrm{E}_7,\qquad\mathrm{E}_8,\qquad \mathrm{F}_4.
	\end{align}
	If the cross section is a rank-$1$ symmetric space of compact type $G/K$, 
	$(M,g)$ is strictly tangentially stable if  $G$ is one of the following real Grasmannians
	\begin{equation}
	\begin{aligned}
	&\frac{\mathrm{SO}(2q+2p+1)}{\mathrm{SO}(2q+1)\times \mathrm{SO}(2p)}\text{ }(p\geq 2,q\geq 1),\qquad
	\frac{\mathrm{SO}(8)}{\mathrm{SO}(5)\times\mathrm{SO}(3)},\\
	&\frac{\mathrm{SO}(2p)}{\mathrm{SO}(p)\times \mathrm{SO}(p)}\text{ }(p\geq 4),\qquad
	\frac{\mathrm{SO}(2p+2)}{\mathrm{SO}(p+2)\times \mathrm{SO}(p)}\text{ }(p\geq 4),\\
	&\frac{\mathrm{SO}(2p)}{\mathrm{SO}(2p-q)\times \mathrm{SO}(q)}\text{ }(p-2\geq q\geq 3),
	\end{aligned}
	\end{equation}
	or one of the following spaces:
	\begin{equation}
	\begin{aligned}
	\mathrm{SU}(2p)/\mathrm{SO}(p)\text{ }(n\geq 6),\qquad
	&\mathrm{E}_6/[\mathrm{Sp}(4)/\left\{\pm I\right\}],\qquad \quad
	\mathrm{E}_6/\mathrm{SU}(2)\cdot \mathrm{SU}(6),\\
	\mathrm{E}_7/[\mathrm{SU}(8)/\left\{\pm I\right\}],\qquad&
	\mathrm{E}_7/\mathrm{SO}(12)\cdot\mathrm{SU}(2),\qquad
	\mathrm{E}_8/\mathrm{SO}(16),\\
	\mathrm{E}_8/\mathrm{E}_7\cdot \mathrm{SU}(2),\qquad&
	\mathrm{F}_4/Sp(3)\cdot\mathrm{SU}(2).
	\end{aligned}
	\end{equation}
	\end{thm}

We also study examples of compact manifolds with isolated conical singularities
where the integrability condition is satisfied. This includes flat spaces with orbifold singularities
as well as K\"ahler manifolds. More precisely we establish the following result. 

\begin{thm}
	Let $(M,h_0)$ be a Ricci-flat K\"{a}hler manifold where the cross section 
	is either strictly tangentially stable or a space form. 
	Then $h_0$ is linearly stable and integrable.
\end{thm}

This paper is organized as follows. After a discussion of the Lichnerowicz Laplacian 
and its Friedrichs extension in \S \ref{laplace-section}, we proceed with a detailed characterization 
of tangentially stable Einstein manifolds in \S \ref{tangential-stability-section}. In \S \ref{mapping-section}
we review the mapping properties of the heat operator as established in \cite{Ver-Ricci}. In \S
\ref{large-section} we establish exponential large time estimates for the heat operator norms.
\S \ref{projection-section} is devoted to analysis of the integrability condition and classes of manifolds
where this condition is satisfied. We conclude the paper with a proof of our main result in 
\S \ref{long-section}. \medskip

\emph{Acknowledgements:} The second author thanks Jan Swoboda
for important discussions about aspects of Ricci flow. Both authors thank 
the Geometry at Infinity Priority program of the German Research Foundation 
DFG for its financial support and for providing a platform for joint research. 
The authors greatfully acknowledge hospitality of the Mathematical 
Institutes at Hamburg and Oldenburg Universities.

\section{The Lichnerowicz Laplacian on conical manifolds}\label{laplace-section}

Let $(M,h)$ be a compact Ricci-flat space with an isolated conical singularity. Let $S:=\textup{Sym}^2({}^{ib}T^*M)$ be the bundle of symmetric $2$-tensors. The Lichnerowicz Laplacian $\Delta_L: C_0^{\infty}(M,S)\to C_0^{\infty}(M,S)$ is a differential operator of second order, defined as
\begin{align}\label{defLL}
\Delta_L\omega=\Delta \omega-2\mathring{R}\omega,\qquad (\mathring{R}\omega)_{ij}:=R_{iklj}\omega^{kl}.
\end{align}
Here, the rough Laplacian $\Delta= \nabla^*\nabla$ is defined with the sign convention such that its eigenvalues are non-negative and the Riemannian curvature tensor is used with the sign convention such that $\mathring{R}h=h$.
We choose local coordinates $(x,z)$ over the singular neighborhood 
$\cC(F) = (0,1)_x \times F$. In the previous paper \cite{Ver-Ricci} 
we have introduced a decomposition of compactly supported smooth sections $C^\infty_0(\cC(F), 
\textup{Sym}^2({}^{ib}T^*M) \restriction \cC(F))$ 
\begin{equation}
\begin{split}
C^\infty_0(\cC(F), S \restriction \cC(F)) &\to C^\infty_0((0,1), C^\infty(F) 
\times \Omega^1(F) \times \textup{Sym}^2(T^*F)),\\
\w &\mapsto \left( \w(\partial_x, \partial_x), \w (\partial_x, \cdot ), 
\w (\cdot, \cdot)\right),
\end{split}
\end{equation}
where $\Omega^1(F)$ denotes differential $1$-forms on $F$.
Under such a decomposition, the Lichnerowicz Laplace operator 
$\Delta_L$ associated to the singular Riemannian metric $g$ 
attains the following form over $\cC(F)$
\begin{equation}
\Delta_L = - \frac{\partial^2}{\partial x^2} - \frac{n}{x} \frac{\partial}{\partial x}
+ \frac{\square_L}{x^2} + \mathscr{O},
\end{equation}
where $\square_L$ is a differential operator on $C^\infty(F) 
\times \Omega^1(F) \times \textup{Sym}^2(T^*F)$ and the 
higher order term $\mathscr{O} \in x^{-1} \V_b^2$ is a second order
differential operator with one order higher asymptotic behaviour at $x=0$.

\begin{defn}\label{tangential-stability-def}
Let $(F^n,g_F)$ be a closed Einstein manifold\footnote{If $(M,g)$ is a 
Ricci-flat space with an isolated conical singularity,
then the cross section $(F,g_F)$ of the cone is automatically Einstein with 
Einstein constant $(n-1)$.}  with Einstein constant $(n-1)$.
Then $(F^n,g_F)$ is called (strictly) tangentially stable if the tangential operator 
of the Lichnerowicz Laplacian on its cone restricted to tracefree tensors is 
non-negative (resp. strictly positive).
\end{defn}
Let $L^2(M,S)$ be the completion of $C^{\infty}_0(M,S)$ with respect to the natural $L^2$-norm induced by the metric $h$.
The inner product on $L^2(M,S)$ induced by $h$ is denoted by $( \cdot, \cdot )_{L^2}$.
We define the maximal closed extension of $\Delta_L$ in $L^2(M,S)$ with domain
\begin{equation}
\dom(\Delta_{L, \max}) := \{\w \in L^2(M,S) \mid \Delta_L \w \in L^2(M,S)\},
\end{equation}
where $\Delta_L \w$ is defined distributionally in terms of the distribution 
$T$ acting on test functions $\phi \in C^{\infty}_0(M,S)$ by $T(\phi) := ( \w, \Delta_L \phi )_{L^2}$. 
We require that the distribution $T$ in fact arises from some $\eta \in L^2(M,S)$ by 
$T(\phi) = (\eta, \phi)_{L^2}$ and we set $\Delta_L \w:= \eta  \in L^2(M,S)$. \medskip

We may also define the minimal closed extension of $\Delta_L$ in $L^2(M,S)$ as the domain of 
the graph closure of $\Delta_L$ acting on $C^{\infty}_0(M,S)$. More precisely, the minimal domain is defined by
\begin{equation*}
\begin{split}
\dom(\Delta_{L, \min}) := \{\w \in \dom(\Delta_{L, \max}) \mid \exists (\w_n)_{n\in \N} \subset C^{\infty}_0(M,S): \\
\w_n \xrightarrow{n\to \infty} \w, \quad \Delta_L \w_n \xrightarrow{n\to \infty} \Delta_L \w \ \textup{in} \ L^2(M,S)\}.
\end{split}
\end{equation*}

Let $(\lambda, \w_\lambda)$ be the set of eigenvalues and 
corresponding eigentensors of the tangential operator $\square_L$.
By the assumption of tangential stability, $\lambda \geq 0$, and we may define
\begin{equation}\label{nu}
\nu(\lambda) := \sqrt{\lambda + 
\left(\frac{n-1}{2}\right)^2}.
\end{equation}
Standard arguments, see e.g. \cite[Lemma 2.2]{MazVer} or \cite{KLP:FDG}, 
cf. the exposition in \cite{Ver:ZDR}, show that for each $\w \in \dom(\Delta_{L, \max})$
there exist constants $c^\pm_\lambda, \nu(\lambda) \in [0,1)$, 
depending only on $\w$, such that $\w$ admits a partial asymptotic expansion as $x\to 0$
\begin{equation}\label{cone-asymptotics}
\begin{split}
\w & = \sum_{\nu(\lambda) = 0} \left(c^+_{\lambda}(\w) x^{ - \frac{(n-1)}{2}}
+ c^-_{\lambda}(\w) x^{ - \frac{(n-1)}{2}} \log(x) \right) \cdot \omega_\lambda 
\\ & + \sum_{\nu(\lambda) \in (0,1)} \left(c^+_{\lambda}(\w) x^{\nu(\lambda) - \frac{(n-1)}{2}}
+ c^-_{\lambda}(\w) x^{-\nu(\lambda) - \frac{(n-1)}{2}} \right) \cdot \omega_\lambda 
\\ & + \widetilde{\w}, \quad \widetilde{\w} \in \dom(\Delta_{L, \min}).
\end{split}
\end{equation}
All self-adjoint extensions for $\Delta_L$ can be classified by 
boundary conditions on the coefficients in the asymptotic 
expansion of solutions in the maximal domain, see e.g. Kirsten, 
Loya and Park \cite[Proposition 3.3]{KLP:FDG}. In particular we define
a self-adjoint extension of $\Delta_L$ on $C^\infty_0(M,S) \subset L^2(M,S)$ 
with domain
\begin{equation}\label{Friedrichs-domain}
\dom(\Delta_{L}) := \{\w \in \dom(\Delta_{L, \max}) \mid 
c^-_{\lambda}(\w) = 0 \ \textup{for} \ \nu(\lambda) \in [0,1)\}.
\end{equation}

\begin{prop}\label{Friedrichs}
Assume that $(M,g)$ is tangentially stable and that the Lichnerowicz 
Laplacian $\Delta_L$ with domain $C^\infty_0(M,S)$ is bounded from 
below by a constant\footnote{The case of $C=0$ is commonly referred 
to as linear stability in the literature.} $C \in \R$. Then the domain of the Friedrichs 
self-adjoint extension $\Delta_L^{\cF}$ of the Lichnerowicz Laplacian is given by
$\dom(\Delta_{L})$ and $\Delta_L^{\cF}$ is bounded from below by $C$.
\end{prop} 

\begin{proof}
Existence of the Friedrichs extension $\Delta_L^{\cF}$ with the same lower
bound as the symmetric densely defined $\Delta_L$ is due to Friedrichs and 
Stone, see Riesz and Nagy \cite[Theorem on p. 330]{RN}. The fact that the 
domain of $\Delta_L^{\cF}$ is given by $\dom(\Delta_{L})$ follows by localizing 
near the conical singularity and using the characterization of the Friedrichs domain 
in \cite[Lemma 3.1 (1)]{BL-cone} as well as e.g. \cite[Corollary 2.14]{Ver:ZDR}.
\end{proof}

\begin{defn}\label{linear-stability}
We say that $(M,g)$ is linearly stable if the the Lichnerowicz 
Laplacian $\Delta_L$ with domain $C^\infty_0(M,S)$ is non-negative.
\end{defn}

Later on, we drop the upper index $\cF$ from notation and 
denote the Friedrichs self-adjoint extension by $\Delta_L$ again. 
Moreover, let us point out that the arguments, constructions 
and definitions extend to compact spaces with finitely many isolated
conical singularities.

\section{Tangential stability of conical manifolds}\label{tangential-stability-section}

In this section, we aim to characterize (strict) tangential stability in terms of eigenvalues of geometric operators on the cross-section of a cone.
In the theorem below, $\Delta_E$ denotes the Einstein operator on symmetric two-tensors over $F$, which is given by $\Delta_E=\nabla^*\nabla-2\mathring{R}$, where $\Delta= \nabla^*\nabla$ is the rough Laplacian on $F$ and $\mathring{R}$ is defined as in 
\eqref{defLL} in terms of the curvature operator of $(F,g_F)$. We write $\Delta$ for the Laplace Beltrami operator on $F$. 
Moreover, $TT$ denotes the space of symmetric two-tensors which are trace-free and divergence-free at each point.
\begin{thm}
	Let $(F,g_F)$, $n\geq 3$ be a compact Einstein manifold with constant $n-1$. Then $(F,g_F)$ is tangentially stable if and only if $\mathrm{Spec}(\Delta_E|_{TT})\geq0$ and $\mathrm{Spec}(\Delta)\setminus \left\{0\right\}\cap (n,2(n+1))=\varnothing$. Similarly, $(M,g)$ is strictly tangentially stable if and only if $\mathrm{Spec}(\Delta_E|_{TT})>0$ and $\mathrm{Spec}(\Delta)\setminus \left\{0\right\}\cap [n,2(n+1)]=\varnothing$. 
\end{thm}
\begin{proof}
	In order to analyse the tangential operator of the Lichnerwicz Laplacian, we use the decomposition of symmetric two-tensors on a Ricci-flat cone that was established in \cite{Kro17}. For the rest of this section, we use the notation in \cite[Section 2]{Kro17} and and the calculations in Section 3.1 of the same paper where we remove all terms containing radial-derivatives in order to obtain expressions for the tangential operator.
More precisely, we write
	\begin{equation}
	\begin{split}
	\{h_i\} \ &\textup{- basis of} \ L^2(TT), \quad \Delta_E h_i = \kappa_i h_i, \quad V_{1,i} := \langle r^2 h_i \rangle, \\
	\{\omega_i\} \ &\textup{- basis of coclosed sections} \ L^2(T^*F), \quad \Delta \omega_i = \mu_i \omega_i, \\ 
	&\quad V_{3,i} :=  \langle r^2 \delta^* \omega_i\rangle \oplus \langle dr \odot r \omega_i\rangle, \\
	\{v_i\} \ &\textup{- basis of} \ L^2(F), \quad \Delta v_i = \lambda_i v_i, \\
	&\quad V_{4,i} := \langle r^2 (n \nabla^2 v_i + \Delta v_i g)\rangle \oplus \langle dr \odot r \nabla v_i\rangle 
	\\ & \quad \oplus \langle v_i (r^2 g - ndr \otimes dr)\rangle.
	\end{split}
	\end{equation}
	Here $\langle\rangle$ denotes the $L^2$-span of a sequence of tensors and $\omega\odot\overline{\omega}:=\omega\otimes\overline{\omega}+\overline{\omega}\otimes\omega$ is the symmetric tensor product.
	Moreover, $\Delta$ in $\Delta \omega_i$ denotes the connection Laplacian, while $\Delta$ in $\Delta v_i$ 
	denotes the Laplace Beltrami operator.
	The spaces $V_{1,i}, V_{3,i}, V_{4,i}$, with $L^2(0,1)$ coefficients, span all trace-free sections $L^2(S_0 \restriction F)$ over $F$, 
	and are invariant under the action of the Lichnerowicz Laplacian.
	At first, if $\tilde{h} = r^2 h_i \in V_{1,i}$,
	\begin{align*}
	(\square_L\tilde{h},\tilde{h})_{L^2}=\kappa_i\Vert \tilde{h} \Vert_{L^2},
	\end{align*}
	such that $\square_L$ is positive (non-negative) on $V_{1,i}$ for all $i$ if and only if all eigenvalues $\kappa_i$ of the Einstein operator on $TT$-tensors are positive (non-negative).\\
	Let $\tilde{h} = \tilde{h}_1 + \tilde{h}_2 = \varphi r^2 \delta^* \omega_i + \psi dr \odot r \omega_i \in V_{3,i}$ with $\varphi,\psi\in\R$. In this case, we have the scalar products
	\begin{align*}	(\square_L \tilde{h}_1, \tilde{h}_1)_{L^2} &= \frac{\varphi^2}{2} (\mu_i - (n-1))^2,\\
	(\square_L \tilde{h}_2, \tilde{h}_2)_{L^2}& = \psi^2 [2\mu_i + (2n+6)],\\
	(\square_L \tilde{h}_1, \tilde{h}_2)_{L^2}& = -2 (\mu_i - (n-1)) \psi \varphi.
	\end{align*}
	Taking $r^2 \delta^* \omega_i$ and $dr \odot r \omega_i$ as a basis, $\square_L$ respects the subspace and acts as $2 \times 2$-matrix
	\begin{equation*}
	\left( \begin{array}{cc}
	\frac{1}{2} (\mu_i - (n-1))^2 & -2(\mu_i - (n-1))\\
	-2(\mu_i - (n-1)) & 2 \mu_i + (2n + 6)
	\end{array} \right).
	\end{equation*}
	We obtain
	\begin{align*}
	\vert 4 (\mu_i - (n-1)) \varphi \psi \vert & = \vert 2 \cdot \underbrace{\frac{1}{\sqrt{2+\epsilon}} (\mu_i - (n-1)) \varphi}_{= a} \cdot \underbrace{\sqrt{2+\epsilon} \cdot 2 \psi}_{= b} \vert\\
	& \le a^2 + b^2 = \frac{1}{2+\epsilon} (\mu_i - (n-1)) \varphi^2 + 4(2+\epsilon) \psi^2,
	\end{align*}
	and therefore,
	\begin{align*}
	(\square_L \tilde{h},\tilde{h})_{L^2} & \ge \varphi^2(\mu_i - (n-1)) [\frac{1}{2} - \frac{1}{2+\epsilon}] + \psi^2 [2 \mu_i + (2n + 6) - 8 -4 \epsilon]\\
	& = 2[\frac{1}{2} - \frac{1}{2+\epsilon}] \cdot \Vert \tilde{h}_1 \Vert_{L^2}^2 + [\mu_i + n + 3-4-3\epsilon] \Vert \tilde{h}_2 \Vert_{L^2}^2 \\ &\ge C(\mu_i) \Vert \tilde{h} \Vert_{L^2}^2,
	\end{align*}
	with  $C(\mu_i)>0$ because $\mu_i>0$ for all $i$ since $n\geq 2$. We also used that 
	\begin{equation*}
	\begin{split}
	\Vert \tilde{h}_1 \Vert_{L^2}^2 & = \frac{1}{2} (\mu_i - (n-1)) \cdot \vert \varphi \vert^2,\qquad 
	\Vert \tilde{h}_2 \Vert_{L^2}^2  = 2 \vert \psi \vert^2. 
	\end{split}
	\end{equation*}
	Therefore, $\square_L$ is always strictly positive on the spaces $ V_{3,i}$.
	It remains to consider the case 
	\begin{align*}\tilde{h} = \tilde{h}_1 + \tilde{h}_2 + \tilde{h}_3 = \varphi r^2 (n \nabla^2 v_i + \Delta v_i g) + \psi dr \odot r \nabla v_i + \mathcal{X} v_i (r^2 g - ndr \otimes dr) \in V_{4,i},
	\end{align*}
	with $\varphi,\psi,\mathcal{X}\in\R$
	which is the most delicate one. We have the scalar products	\begin{equation*}
	\begin{split}
	(\square_L \tilde{h}_1,\tilde{h}_1)_{L^2} & = n(n-1) \lambda_i (\lambda_i-n) (\lambda_i - 2(n-1)) \varphi^2,\\
	(\square_L \tilde{h}_2,\tilde{h}_2)_{L^2} & = [2\lambda_i (\lambda_i - (n-1)) + (2n+6) \lambda_i] \psi^2,\\
	(\square_L \tilde{h}_3,\tilde{h}_3)_{L^2} & = [n\{(n+1) \lambda_i - 2(n+1) \} + 2n^2(n+3)] \mathcal{X}^2,\\
	(\square_L \tilde{h}_1,\tilde{h}_2)_{L^2} & = -4(n-1) \lambda_i (\lambda_i -n) \psi \varphi,\\
	(\square_L \tilde{h}_2,\tilde{h}_3)_{L^2} & = 4(n+1) \lambda_i \psi \mathcal{X},\\
	(\square_L \tilde{h}_1,\tilde{h}_3)_{L^2} & = 0,
	\end{split}
	\end{equation*}
	and the norms
	\begin{equation*}
	\begin{split}
	\Vert \tilde{h}_1 \Vert_{L^2}^2 & = n(n-1) \lambda_i (\lambda_i - n) \varphi^2,\\
	\Vert \tilde{h}_2 \Vert_{L^2}^2 & = 2 \varphi^2 \lambda_i,\\
	\Vert \tilde{h}_3 \Vert_{L^2}^2 & = (n+1)n.
	\end{split}
	\end{equation*}
	Consider $(\square_L - \epsilon I)$. It acts as a matrix $A=(a_{ij})_{1\leq n\leq 3}$,
	whose coefficients are given by
	\begin{align*}
	a_{11}&=n(n-1)\lambda_i (\lambda_i - n)[\lambda_i -2(n-1)-\epsilon],\\
	a_{22}&=2 \lambda_i [\lambda_i - (n-1) -\epsilon + n + 3],\\
	a_{33}&= n\{(n+1) \lambda_i - 2(n-1) - \epsilon(n+1) + 2n(n+3) \},\\
	a_{12}&=a_{21}=-4(n-1)\lambda_i (\lambda_i -n),\\
	a_{23}&=a_{32}=4(n+1)\lambda_i,\\
	a_{13}&=a_{31}=0.
	\end{align*}
	In order to prove positivity (resp.) nonnegativity of this matrix, we consider its principal minors $A_{33}$ (which is the lower right entry), $A_{23}$ ( the lower right $2\times 2$-matrix) and $A$ ( the whole matrix).
	At first,
	\begin{equation*}
	\begin{split}
	A_{33} & = n\{(n+1) \lambda_i - 2 (n-1) - \epsilon (n+1) + 2 n (n+3) \}\\
	& = n\{(n+1)\lambda_i + 2n^2 + 6n - 2n + 2 - \epsilon n - \epsilon \}\\
	& = n \{(n+1) \lambda_i + 2n^2 + (4-\epsilon) n + (2-\epsilon) \} > 0
	\end{split}
	\end{equation*}
	for any $\epsilon < 2$. Observe that in the case $\lambda_i=0$, $\tilde{h}_1 \equiv 0$ and $ \tilde{h}_2 \equiv 0$, so that $V_{4i} = span\{\tilde{h}_3\}$ and
	hence, $(\square_L-\epsilon I)$ acts as $A_{33} > 0$ for $\epsilon< 2$. Therefore, we may from now on assume that $\lambda_i > 0$, which means that actually $\lambda_i \ge n$ (due eigenvalue estimates for Einstein manifolds, see e.g. \cite{Ob62}) with $\lambda_i = n$ only for $\mathbb{S}^n$. By considering the matrix
	\begin{equation*}
	\left( \begin{array}{cc}
	2[\lambda_i + 4 - \epsilon] & 4(n+1)\lambda_i\\
	4(n+1) & n\{(n+1) \lambda_i + 2n^2 + (4-\epsilon)n + (2-\epsilon) \}
	\end{array} \right),
	\end{equation*}
	from which one recovers $A_{23}$ by multiplying the first column by $\lambda_i$,
	we see that
	\begin{equation*}
	\begin{split}
	\frac{\det A_{23}}{\lambda_i} & = 2[\lambda_i + 4 - \epsilon]n \cdot [(n+1)\lambda_i + \underbrace{2n^2 + (3-\epsilon)n + (2-\epsilon)}_{=p(n)}] - 16(n+1)^2 \lambda_i\\
	& = 2[\lambda_i + (4-\epsilon)] [(n^2 + n) \lambda_i + n p(n)] - 16(n^2 + 2n +1) \lambda_i\\
	& = 2\lambda_i^2 (n^2 + n) + 2 \lambda_i n p(n) + (8-2\epsilon) (n^2 + n) \lambda_i - 16(n^2+2n+1)\lambda_i\\
	& \qquad + (8-2\epsilon) n p(n)\\
	& = 2(n^2+n) \lambda_i^2 + [2np(n) + (8-2\epsilon)(n^2+n) - 16(n^2+2n+1)] \lambda_i\\
	& \qquad + (8-2\epsilon) n p(n)\\
	& = \underbrace{2n(n+1)}_{>0} \lambda_i^2 + [2np(n) + (n+1) \{(8-2\epsilon)n-16(n+1)\}] \lambda_i\\ &\qquad + \underbrace{(8-2\epsilon)np(n)}_{>0}.
	\end{split}
	\end{equation*}
	The $\lambda_i$-coefficient satisfies
	\begin{equation*}
	\begin{split}
	2np(n) + (n+1)\{(8-2\epsilon)n - 16n -16 \} & = 4n^3 - 4\epsilon n^2 - 20 n - 4 \epsilon n -16\\
	& = 4(n+1) (n^2-n-4) -4 \epsilon n (n+1)\\
	& = 4(n+1) [n^2 - (1+\epsilon) n -4]\\
	& > 0
	\end{split}
	\end{equation*}
	for $n \ge 3$ and $\epsilon$ sufficiently small. Therefore, we obtain $\det A_{23}>0$ in these cases. For $n=2$ we compute explicitly
	\begin{align*}
	\lambda_i^{-1} \det A_{23} & =	12 \lambda_i^2 + [12(-(1+\epsilon)\cdot 2)] \lambda_i + (16-2\epsilon + 2 - \epsilon) \cdot (8-2\epsilon) \cdot 2\\ & = 12\lambda_i^2-24(1+\epsilon) \lambda_i + 16\cdot 18 + \mathcal{O}(\epsilon).
	\end{align*}
	Note that $12x^2-24x+16\cdot 18$ has no zeros, so that for $\epsilon$ sufficiently small the expression is always positive.
Before we compute the full determinant of $A$, we remark that in the case $\lambda_i=n$, the tensor $\tilde{h}_1$ is vanishing so that in this case, the matrix $A$ describing $\square_L$ on $V_{4,i}$ reduces to the matrix $A_{23}$ which just has been considered. Therefore, there is nothing more to prove in this case and we may assume $\lambda=\lambda_i>0$ from now on.
\medskip
	
	To compute the full determinant of $A$, we first consider the matrix
	\begin{align*}
	\left( \begin{array}{ccc}
	n[\lambda - 2(n-1) - \epsilon] & -2(n-1)(\lambda - n) & 0\\
	-4 & \underbrace{\lambda-(n-1)-\epsilon+n+3}_{=\lambda + 4-\epsilon} & 4 \lambda\\
	0& 2(n+1) & n\{\lambda + 2(n+1) - \epsilon \} 
	\end{array} \right) ,
	\end{align*}
	from which we recover $A$ by mutiplying the three columns by $(n-1)\lambda(\lambda-n)$ and $2\lambda,(n+1)$, respectively.
	We get
	\begin{equation*}
	\begin{split}
	[(n-1) \lambda (\lambda -n) 2 \cdot &\lambda \cdot (n+1)]^{-1} \det A \\
	& = n[\lambda - 2(n-1) - \epsilon] [(\lambda + 4 - \epsilon) \cdot n (\lambda + 2(n+1) - \epsilon)\\
	& \qquad - 8 \lambda (n+1)] + 4\{-2 (n-1)(\lambda -n) \cdot n(\lambda + 2(n+1) - \epsilon) \}\\
	& = -4 n^4 \lambda + 4 \epsilon n^4 + 8n^3 \lambda - 8 \epsilon n^3 + n^2 \lambda^3 - 3\epsilon n^2 \lambda^2 - 8 n^2 \lambda^2\\
	&\qquad + 3\epsilon^2n^2\lambda + 12n^2\lambda - \epsilon^3n^2 + 8\epsilon^2n^2-12\epsilon n^2\\
	& = \lambda^3 n^2 + \lambda^2[-3\epsilon n^2 - 8 n^2]\\
	& \qquad + \lambda[-4n^4 + 8n^3+3 \epsilon^2 n^2 + 12n^2] + 4 \epsilon n^4 - 8 \epsilon n^3 -\epsilon^2n^2\\
	& \qquad + 8\epsilon^2 n^2 - 12 \epsilon n^2\\
	& = \lambda^3 n^2 - 8 n^2 \lambda^2 + \lambda[-4n^4 + 8n^3 + 12n^2] + \mathcal{O}(\epsilon)\\
	& = \lambda n^2 [\lambda^2 - 8 \lambda - 4n^2 + 8n + 12] + \mathcal{O}(\epsilon)\\
	& = \lambda n^2(\lambda-2n-2)(\lambda+2n-6) + \mathcal{O}(\epsilon).\\	
	\end{split}
	\end{equation*}
	We get positivity of $\det(A)$ if and only if $\lambda>2n+2$ for all positive eigenvalues $\lambda$. Moreover, $\det(A)=0$ for $\epsilon>0$ if $\lambda=2n+2$. Due to positivity of the determinants of the other principal minors in this case, $A$ is positive semidefinite if $\epsilon=0$ and $\lambda=2n+2$.
\end{proof}
\begin{ex}
	Any spherical space form is tangentially stable: We have $\Delta_E|_{TT}>0$ in this case due to an unpublished result by Bourguignon (see e.g. \cite[Corollary 12.72]{Besse}). Moreover, $\mathrm{Spec}(\Delta)\setminus \left\{0\right\}\cap (n,2(n+1))=\varnothing$ holds for the sphere and this property also descents to any of its quotients. The spaces $\mathbb{S}^n$ and $\mathbb{RP}^n$ are not strictly tangentially stable as $2(n+1)$ is contained in their spectrum. This property may also hold for other spherical space forms.
\end{ex}
\begin{thm}
	Let $(F^n,g_F)$, $n\geq 2$ be a closed Einstein manifold with constant $n-1$, which is a symmetric space of compact type. If it is a simple Lie group $G$, it is strictly tangentially stable if $G$ is one of the following spaces:
	\begin{align}
	\mathrm{Spin}(p)\text{ }(p\geq 6,p\neq 7),\qquad \mathrm{E}_6,\qquad\mathrm{E}_7,\qquad\mathrm{E}_8,\qquad \mathrm{F}_4.
	\end{align}
	On the other hand, it is tangentially unstable, if $G$ is one of the following spaces:
	\begin{align}
	\mathrm{SU}(p+1)\text{ }(p\geq 3),\qquad \mathrm{Spin}(5),\qquad \mathrm{Spin}(7),\qquad\mathrm{Sp}(p)\text{ }(p\geq 3),\qquad\mathrm{G}_2.
	\end{align}
	If the manifold $(F,g_F)$ is a rank-1 symmetric space $G/K$ of compact type, it is strictly tangentially stable if  it is one of the following real Grasmannians
	\begin{equation}
	\begin{aligned}
	&\frac{\mathrm{SO}(2q+2p+1)}{\mathrm{SO}(2q+1)\times \mathrm{SO}(2p)}\text{ }(p\geq 2,q\geq 1),\qquad
	\frac{\mathrm{SO}(8)}{\mathrm{SO}(5)\times\mathrm{SO}(3)},\\
	&\frac{\mathrm{SO}(2p)}{\mathrm{SO}(p)\times \mathrm{SO}(p)}\text{ }(p\geq 4),\qquad
	\frac{\mathrm{SO}(2p+2)}{\mathrm{SO}(p+2)\times \mathrm{SO}(p)}\text{ }(p\geq 4),\\
	&\frac{\mathrm{SO}(2p)}{\mathrm{SO}(2p-q)\times \mathrm{SO}(q)}\text{ }(p-2\geq q\geq 3),
	\end{aligned}
	\end{equation}
	or one of the following spaces:
	\begin{equation}
	\begin{aligned}
	\mathrm{SU}(2p)/\mathrm{SO}(p)\text{ }(p\geq 6),\qquad
	&\mathrm{E}_6/[\mathrm{Sp}(4)/\left\{\pm I\right\}],\qquad \quad
	\mathrm{E}_6/\mathrm{SU}(2)\cdot \mathrm{SU}(6),\\
	\mathrm{E}_7/[\mathrm{SU}(8)/\left\{\pm I\right\}],\qquad&
	\mathrm{E}_7/\mathrm{SO}(12)\cdot\mathrm{SU}(2),\qquad
	\mathrm{E}_8/\mathrm{SO}(16),\\
	\mathrm{E}_8/\mathrm{E}_7\cdot \mathrm{SU}(2),\qquad&
	\mathrm{F}_4/Sp(3)\cdot\mathrm{SU}(2).
	\end{aligned}
	\end{equation}
	On the other hand it is unstable if $G/K$ is $\mathbb{CP}^p, \mathbb{HP}^p$, $p\geq2$,
	one of the (real, complex and quaternionic) Grasmannians
	\begin{equation}
	\begin{aligned}
	&\frac{\mathrm{SO}(2p+2)}{\mathrm{SO}(2p)\times \mathrm{SO}(2)}\text{ }(p\geq 3),\qquad\quad
	\frac{\mathrm{SO}(5)}{\mathrm{SO}(3)\times\mathrm{SO}(2)},\\
	&\frac{\mathrm{SO}(2p+3)}{\mathrm{SO}(2p+1)\times \mathrm{SO}(2)}\text{ }(p\geq 2),\qquad
	\frac{\mathrm{U}(q+p)}{\mathrm{U}(q)\times \mathrm{U}(p)}\text{ }(q\geq p\geq 2),\\
	&\frac{\mathrm{Sp}(q+p)}{\mathrm{Sp}(q)\times \mathrm{Sp}(p)}\text{ }(q \geq p\geq 2),
	\end{aligned}
	\end{equation}
	or one of the following spaces:
	\begin{equation}
	\begin{aligned}
	\mathrm{SU}(2p)/\mathrm{SO}(p)\text{ }(5\geq p\geq 3),\quad
	&\mathrm{SU}(2p)/\mathrm{Sp}(p)\text{ }(p\geq 3),\quad
	\mathrm{Sp}(p)/\mathrm{U}(p)\text{ }(p\geq 3),\\
	\mathrm{SO}(2p)/\mathrm{U}(p)\text{ }(p\geq 5),\quad
	&\mathrm{E}_6/\mathrm{SO}(10)\cdot \mathrm{SO}(2),\quad
	\mathrm{E}_6/\mathrm{F}_4,\\
	\mathrm{E}_7/\mathrm{E}_6\cdot \mathrm{SO}(2),\quad&
	\mathrm{F}_4/\mathrm{Spin}(9),\quad
	\mathrm{G}_2/\mathrm{SO}(4).
	\end{aligned}
	\end{equation}
\end{thm}
\begin{remark}This theorem shows that the sphere is the only example in this class which is tangentially stable but not strictly tangentially stable.
\end{remark}
\begin{proof}
	We analyse the tables $2$ and $3$ in \cite{Kro17b}. In table 2, we have to check, which of the Lie groups $G$ are (strictly) stable (which means $\Delta_E|_{TT}\geq0$, resp.\ $\Delta_E|_{TT}>0$) and for which all non-zero eigenvalues $\lambda$ of the Laplacian satisfy
	the condition $\lambda\geq 2(\dimn(G)+1)$ (resp.\ $>$) or equivalently,
	\begin{align*}
	\Lambda\geq2\frac{\dimn(G)+1}{\dimn(G)-1}
	\end{align*}
	(resp.\ $>$). Here, $\Lambda$ is the eigenvalue $\lambda$ normalized by the Einstein constant $\dimn(G)-1$. By checking these conditions, we obtain the following table:
	\begin{center}
		\renewcommand{\arraystretch}{1.5}
		\begin{longtable}{|l|l|c|c|l|l|}
			\hline
			type & $\mathrm{G}$ & $\dimn(\mathrm{G})$ & $\Lambda$ & stability & tang.\ stability \\
			\hline
			$\mathrm{A}_p$	& $\mathrm{SU}(p+1)$, $p\geq 2$			& $p^2-1$			
			& $\frac{2p(p+2)}{(p+1)^2}$ & unstable & unstable\\
			\hline
			\multirow{3}{*}{$\mathrm{B}_n$}
			& $\mathrm{Spin}(5)	$		& $10$		& $\frac{5}{3}$ & unstable & unstable\\
			& $\mathrm{Spin}(7)	$		& $21$	& $\frac{21}{10}$ & s.\ stable & unstable \\
			& $\mathrm{Spin}(2p+1)$, $n\geq 4$		& $2p(p+1)$		& $\frac{4p}{2p-1}$ & s.\ stable & s.\ stable\\
			\hline
			$\mathrm{C}_p$	& $\mathrm{Sp}(p)$, $p\geq 3$			& $p(2p+1)$		& $\frac{2p+1}{p+1}$ & unstable & unstable\\
			\hline
			$\mathrm{D}_p$	& $\mathrm{Spin}(2p)$, $p\geq 3$			& $p(2p+1)$		& $\frac{2p-1}{p-1}$ & s.\ stable & s.\ stable\\
			\hline
			$\mathrm{E}_6$	& $\mathrm{E}_6$			& $156$		& $\frac{26}{9}$ & s.\ stable & s.\ stable\\
			\hline
			$\mathrm{E}_7$	& $\mathrm{E}_7$			& $266$		& $\frac{19}{6}$ & s.\ stable & s.\ stable\\
			\hline
			$\mathrm{E}_8$	& $\mathrm{E}_8$			& $496$		& $4$ & s.\ stable & s.\ stable\\
			\hline
			$\mathrm{F}_4$	& $\mathrm{F}_4$			& $52$		& $\frac{8}{3}$ & s.\ stable & s.\ stable\\
			\hline
			$\mathrm{G}_2$	& $\mathrm{G}_2$			& $14$		& $2$ & stable & unstable\\
			\hline	
			\caption{Tangential stability of simple Lie groups}
		\end{longtable}
	\end{center}
	In the case of irreducible rank-$1$ symmetric spaces of compact type, an analogous argumentation yields the following table:
	\begin{center}
		\renewcommand{\arraystretch}{1.5}
		\begin{longtable}{|l|l|c|c|l|l|}
			\hline
			type & $G/K$ & $\dimn(G/K)$ & $\Lambda$ & stability & tang. stab. \\
			\hline	
			\multirow{2}{*}{A I}	& $\mathrm{SU}(p)/\mathrm{SO}(p)$, $5\geq p\geq 3$			
			& $\frac{(p-1)(p+2)}{2}$			& $\frac{2(p-1)(p+2)}{p^2}$ &  stable &  unstable\\
			& $\mathrm{SU}(p)/\mathrm{SO}(p)$, $p\geq 6$			
			& $\frac{(p-1)(p+2)}{2}$			& $\frac{2(p-1)(p+2)}{p^2}$ &  stable &  s.\ stable\\
			\hline
			\multirow{2}{*}{A II}
			& $\mathrm{SU}(4)/\mathrm{Sp}(2)=S^5	$		& $5$		& $\frac{5}{4}$ & s.\ stable & stable \\
			& $\mathrm{SU}(2p)/\mathrm{Sp}(p)$, $p\geq3	$		& $2p^2-p-1$	& $\frac{(2p+1)(p-1)}{p^2}$ & unstable & unstable \\
			\hline
			\multirow{2}{*}{A III}
			& $\frac{\mathrm{U}(p+1)}{\mathrm{U}(p)\times \mathrm{U}(1)}=\CP^p	$		& $2p$		& $2$ & stable & unstable\\
			& $\frac{\mathrm{U}(p+q)}{\mathrm{U}(q)\times \mathrm{U}(p)}$, $q\geq p\geq2	$		
			& $2pq$	& $2$ &  stable & unstable \\
			\hline
			\pagebreak 
			\hline
			type & $G/K$ & $\dimn(G/K)$ & $\Lambda$ & stability & tang. stab. \\
			\hline
			\multirow{6}{*}{B I}
			& $\frac{\mathrm{SO}(5)}{\mathrm{SO}(3)\times SO(2)}$		& $6$	& $2$ &  unstable & unstable \\		
			& $\frac{\mathrm{SO}(2p+3)}{\mathrm{SO}(2p+1)\times \mathrm{SO}(2)}$, $p\geq 2$		& $4p+2$	& $2$ &  stable & unstable \\
			& $\frac{\mathrm{SO}(7)}{\mathrm{SO}(4)\times \mathrm{SO}(3)}$		& $12$	& $\frac{12}{5}$ & s.\ stable & s.\ stable \\
			& $\frac{\mathrm{SO}(2p+3)}{\mathrm{SO}(3)\times \mathrm{SO}(2p)}$, $p\geq 3$		
			& $6p$	& $\frac{4p+6}{2p+1}$ & s.\ stable & s.\ stable \\
			& $\frac{\mathrm{SO}(2q+2p+1)}{\mathrm{SO}(2q+1)\times \mathrm{SO}(2p)}$, $p,q\geq 2$		
			& $2n(2m+1)$	& $\frac{4m+4n+2}{2m+2n-1}$ & s.\ stable & s.\ stable \\
			\hline
			\multirow{1}{*}{B II} & $\frac{\mathrm{SO}(2p+1)}{\mathrm{SO}(2p)}=\mathbb{S}^{2p}$, $p\geq 1$		
			& $2p$		& $\frac{2p}{2p-1}$ & s.\ stable & stable \\	
			\hline
			C I	& $\mathrm{Sp}(p)/\mathrm{U}(p)$, $p\geq 3$		& $p(p+1)$		& $2$ & unstable & unstable\\		
			\hline
			\multirow{3}{*}{C II}
			& $\frac{\mathrm{Sp}(2)}{\mathrm{Sp}(1)\times \mathrm{Sp}(1)}=\mathbb{S}^4	$		
			& $4$		& $\frac{4}{3}$ & s.\ stable & stable \\
			& $\frac{\mathrm{Sp}(p+1)}{\mathrm{Sp}(p)\times \mathrm{Sp}(1)}=\HP^p$, $p\geq 2$		
			& $4p$	& $\frac{2(p+1)}{p+2}$ & unstable & unstable \\
			& $\frac{\mathrm{Sp}(p+q)}{\mathrm{Sp}(q)\times \mathrm{Sp}(p)}$, $q\geq p\geq 2$		
			& $4pq$	& $\frac{2(p+q)}{p+q+1}$ & unstable & unstable \\
			\hline
			\multirow{6}{*}{D I}
			
			& $\frac{\mathrm{SO}(8)}{\mathrm{SO}(5)\times\mathrm{SO}(3)}$		& $15$	& $\frac{5}{2}$ &  s.\ stable & s.\ stable \\		
			& $\frac{\mathrm{SO}(2p+2)}{\mathrm{SO}(2p)\times \mathrm{SO}(2)}$, $p\geq 3$		& $4p$	& $2$ & stable & unstable \\
			& $\frac{\mathrm{SO}(2p)}{\mathrm{SO}(p)\times \mathrm{SO}(p)}$, $p\geq 4$		& $p^2$	& $\frac{2n}{n-1}$ & s.\ stable & s.\ stable \\
			& $\frac{\mathrm{SO}(2p+2)}{\mathrm{SO}(p+2)\times \mathrm{SO}(p)}$, $p\geq 4$		& $p(p+2)$	& $\frac{2p+2}{p}$ & s.\ stable & s.\ stable \\
			& $\frac{\mathrm{SO}(2p)}{\mathrm{SO}(2p-q)\times \mathrm{SO}(q)}$,	& $(2p-q)q$	& $\frac{2p}{p-1}$ & s.\ stable & s.\ stable \\
			& $p-2\geq q\geq 3$	& & & & \\
			\hline
			D II& $\frac{SO(2p+2)}{SO(2p+1)}=S^{2p+1}$, $p\geq 3$		& $2p+1$		& $\frac{2p+1}{2p}$ & s.\ stable & stable \\
			\hline
			D III	& $\mathrm{SO}(2p)/\mathrm{U}(p)$, $p\geq 5$		& $p(p-1)$		& $2$ &  stable & unstable\\
			\hline
			E I	& $\mathrm{E}_6/[\mathrm{Sp}(4)/\left\{\pm I\right\}]$ & $42$		& $\frac{28}{9}$ & s.\ stable & s.\ stable\\
			\hline
			E II	& $\mathrm{E}_6/\mathrm{SU}(2)\cdot \mathrm{SU}(6)$			& $40$		& $3$ & s.\ stable & s.\ stable\\
			\hline
			E III	& $\mathrm{E}_6/\mathrm{SO}(10)\cdot \mathrm{SO}(2)$	& $32$		& $2$ &  stable & unstable\\
			\hline
			E IV	& $\mathrm{E}_6/\mathrm{F}_4$			& $26$		& $\frac{13}{9}$ & unstable & unstable\\
			\hline
			E V	& $\mathrm{E}_7/[\mathrm{SU}(8)/\left\{\pm I\right\}]$ & $70$		& $\frac{10}{3}$ & s.\ stable & s.\ stable \\
			\hline
			E VI	& $\mathrm{E}_7/\mathrm{SO}(12)\cdot\mathrm{SU}(2)$ & $64$		& $\frac{28}{9}$ & s.\ stable & s.\ stable \\
			\hline
		\pagebreak
		\hline
			type & $G/K$ & $\dimn(G/K)$ & $\Lambda$ & stability & tang. stab. \\
				\hline
			E VII	& $\mathrm{E}_7/\mathrm{E}_6\cdot \mathrm{SO}(2)$ & $54$		& $2$ & stable & unstable \\
			\hline
			E VIII	& $\mathrm{E}_8/\mathrm{SO}(16)$ & $128$		& $\frac{62}{15}$ & s.\ stable & s.\ stable \\
			\hline
			E IX	& $\mathrm{E}_8/\mathrm{E}_7\cdot \mathrm{SU}(2)$ & $112$		& $\frac{16}{5}$ & s.\ stable & s.\ stable \\
			\hline
			F I	& $\mathrm{F}_4/Sp(3)\cdot\mathrm{SU}(2)$ & $28$		& $\frac{26}{9}$ & s.\ stable & s.\ stable \\
			\hline
			F II	& $\mathrm{F}_4/\mathrm{Spin}(9)$ & $16$		& $\frac{4}{3}$ & unstable & unstable \\
			\hline
			G	& $\mathrm{G}_2/\mathrm{SO}(4)$ & $8$		& $\frac{7}{3}$ & s.\ stable & unstable \\
			\hline
			\caption{Tangential stability of symmetric spaces of non-group type}
		\end{longtable}
	\end{center}
\end{proof}

\section{H\"older spaces on conical manifolds}\label{spaces-section} \medskip

This section is basically a recap of the corresponding definitions in \cite{Ver-Ricci}
in the case of isolated conical singularities. We introduce H\"older spaces adapted 
to the singular geometry and mapping properties of the corresponding heat operator.
We consider a manifold $(M,g)$ with isolated conical singularities and assume for notational
simplicity that we have just one conical end. All constructions extend easily to the case of 
multiple conical ends.

\begin{defn}\label{hoelder-A}
The H\"older space $\ho(M\times [0,T]), \A\in [0,1),$ consists of functions 
$u(p,t)$ that are continuous on $\overline{M} \times [0,T]$ with finite $\A$-th H\"older 
norm
\begin{align}\label{norm-def}
\|u\|_{\A}:=\|u\|_{\infty} + \sup \left(\frac{|u(p,t)-u(p',t')|}{d_M(p,p')^{\A}+
|t-t'|^{\frac{\A}{2}}}\right) <\infty, 
\end{align}
where the distance function $d_M(p,p')$ between any two points $p,p'\in M$ 
is defined with respect to the conical metric $g$, and in terms of the local coordinates 
$(x,z)$ in the singular neighborhood $\cC(F)$ given equivalently by
\begin{align*}
d_M((x,z), (x',z'))=\left(|x-x'|^2+(x+x')^2|z-z'|^2\right)^{\frac{1}{2}}.
\end{align*}
The supremum is taken over all $(p,p',t) \in M^2 \times [0,T]$\footnote{Finiteness 
of the H\"older norm $\|u\|_{\A}$ in particular implies that $u$ is continuous on the 
closure $\overline{M}$ up to the edge singularity, and the supremum may be taken 
over $(p,p',t) \in \overline{M}^2 \times [0,T]$. Moreover, as explained in 
\cite{Ver-Ricci} we can assume without loss of generality that 
the tuples $(p,p')$ are always taken from within the same coordinate 
patch of a given atlas.}. 
\end{defn} 

We now extend the notion of H\"older spaces to sections of the  
vector bundle $S=\textup{Sym}^2({}^{ib}T^*M)$ of symmetric $2$-tensors. 

\begin{defn}\label{S-0-hoelder}
Denote the fibrewise inner product on $S$ induced by the Riemannian metric $g$, again by $g$.
The H\"older space $\ho (M\times [0,T], S)$ consists of all sections $\w$ of
$S$ which are continuous on $\overline{M} \times [0,T]$, 
such that for any local orthonormal frame 
$\{s_j\}$ of $S$, the scalar functions $g(\w,s_j)$ are $\ho (M\times [0,T])$.
\medskip

The $\A$-th H\"older norm of $\w$ is defined using a partition of unity
$\{\phi_j\}_{j\in J}$ subordinate to a cover of local trivializations of $S$, with a 
local orthonormal frame $\{s_{jk}\}$ over $\supp (\phi_j)$ for each $j\in J$. We put
\begin{align}\label{partition-hoelder-2}
\|\w\|^{(\phi, s)}_{\A}:=\sum_{j\in J} \sum_{k} \| g(\phi_j \w,s_{jk}) \|_{\A}.
\end{align}
\end{defn}

Norms corresponding to different choices of $(\{\phi_j\}, \{s_{jk}\})$
are equivalent and we may drop the upper index $(\phi, s)$ from notation.
We now turn to weighted and higher order H\"older spaces. We extend the boundary
defining function $x:\cC(F) \to (0,1)$ smoothly to a non-vanishing function on $M$.
The weighted H\"older spaces of higher order are now defined as follows. 

\begin{defn}\label{funny-spaces}
\begin{enumerate}
\item The weighted H\"older space for $\gamma \in \R$ is
\begin{align*}
&x^\gamma \ho(M\times [0,T], S) := \{ \, x^\gamma \w \mid \w \in \ho(M\times [0,T], S) \, \}, 
\\ &\textup{with H\"older norm} \ \| x^\gamma \w \|_{\A, \gamma} := \|\w\|_{\A}.
\end{align*}
\item The hybrid weighted H\"older space for $\gamma \in \R$ is
\begin{align*}
&\ho_{, \gamma} (M\times [0,T], S) := x^\gamma \ho(M\times [0,T], S)  \, \cap \, 
x^{\gamma + \A} \mathcal{C}^0_{\textup{ie}}(M\times [0,T], S) \\
&\textup{with H\"older norm} \  \| \w \|'_{\A, \gamma} := \|x^{-\gamma} 
\w\|_{\A} + \|x^{-\gamma-\A} \w\|_\infty.
\end{align*}
\item The weighted higher order H\"older spaces, which specify regularity of solutions 
under application of the Levi Civita covariant derivative $\nabla$ of $g$ on symmetric $2$-tensors
and time differentiation are defined for any 
$\gamma \in \R$ and $k \in \N$ by\footnote{Differentiation is a priori 
understood in the distributional sense.}\footnote{We require regularity of $\w$ under differentiation by $x^2\partial_t$
instead of just $\partial_t$, since in the discussion below, $\partial_t \w \mid_{t=0}$ need not be continuous up to $x=0$.}
\begin{equation*}
\begin{split}
&\hok (M\times [0,T], S)_\gamma = \{\w\in \ho_{,\gamma} \mid  \{\nabla_{\V_b}^j \circ \, (x^2 \partial_t)^\ell\} 
\, \w \in \ho_{,\gamma} \ \textup{for any} \ j+2\ell \leq k \}, \\
&\hok (M\times [0,T], S)^b_\gamma = 
\{u\in \ho \mid  \{\nabla_{\V_b}^j \circ \, (x^2 \partial_t)^\ell\} \, u \in 
x^\gamma\ho \ \textup{for any} \ j+2\ell \leq k\},
\end{split}
\end{equation*}
where $j,l\in\N_0$, the upper index b in the second space indicates the fact that despite
the weight $\gamma$, the solutions $u \in \hok (M\times [0,T], S)^b_\gamma$
are only bounded, i.e. $u\in \ho$. The corresponding H\"older norms are defined 
using local bases $\{X_i\}$ of $\V$ and $\mathscr{D}_k:=\{\nabla_{X_{i_1}} \circ \cdots \circ 
\nabla_{X_{i_j}} \circ (x^2 \partial_t)^\ell \mid j+2\ell \leq k\}$ by
\begin{equation*}
\begin{split}
&\|\w\|_{k+\A, \gamma} = \sum_{j\in J} \sum_{X\in \mathscr{D}_k} \| X (\phi_j \w) \|'_{\A, \gamma}
+ \|\w\|'_{\A, \gamma}, \quad \textup{on} \ \hok (M\times [0,T], S)_\gamma, \\
&\|u\|_{k+\A, \gamma} = \sum_{j\in J} \sum_{X\in \mathscr{D}_k} \| X (\phi_j u) \|_{\A, \gamma}
+ \|u\|_{\A}, \quad \textup{on} \ \hok (M\times [0,T], S)^b_\gamma.
\end{split}
\end{equation*} 
\item In case of $\gamma=0$ we just omit the lower weight index and write
e.g. $\hok (M\times [0,T], S)$ and $\hok (M\times [0,T], S)^b$.
\end{enumerate} \ \medskip

The subspaces of time-independent functions are denoted by 
\begin{equation} \begin{split}
\hok (M, S)_\gamma \subset \hok (M\times [0,T], S)_\gamma,  \\
\hok (M, S)^b_\gamma \subset \hok (M\times [0,T], S)^b_\gamma.
\end{split} \end{equation}
\end{defn} 

The H\"older norms for different choices of local bases $\{X_1, \ldots, X_m\}$ of $\V_b$ and different choices
of Riemannian metrics $g$ with isolated conical singularities, are equivalent due to compactness of $M$ and $F$.
Note also that $|X_i|_g=O(x)$ so that $\omega\in C^{\alpha}_{ie,\gamma}$ implies  $|\partial_t^l\nabla^j\omega|_g=O(x^{\gamma+\alpha-j-2l})$ for $j,l\in\N_0$. Such spaces are very natural in the conical setting as elliptic regularity and Fredholm theory of elliptic operaters defined on these spaces is avaiable.

 \medskip

The vector bundle $S$ decomposes into a direct sum of sub-bundles
\begin{align}
S= S_0 \oplus S_1, 
\end{align}
where the sub-bundle $S_0=\textup{Sym}_0^2({}^{ib}T^*M)$
is the space of trace-free (with respect to the fixed metric $g$) symmetric $2$-tensors,
and $S_1$ is the space of pure trace (with respect to the fixed metric $g$) symmetric 
$2$-tensors. The sub bundle $S_1$ is trivial real vector bundle over $M$ of rank 1.
\medskip

Definition \ref{funny-spaces} extends ad verbatim to sections of $S_0$ and $S_1$.
Since the sub-bundle $S_1$ is a trivial rank one real vector bundle, its sections
correspond to scalar functions. In this case we may omit $S_1$ from the notation and
simply write e.g. $\hok (M\times [0,T])^b_\gamma$. \medskip

\begin{remark}
The higher order weighted H\"older spaces in Definition \ref{funny-spaces}
differ slightly from the corresponding spaces in \cite{Ver-Ricci} by the choice of 
admissible derivatives. While in \cite{Ver-Ricci} we allow differentiation by any 
b-vector field $\V \in \V_b$, here we employ only derivatives of the form $\nabla_{\V},
\V \in \V_b$. \medskip
\end{remark}

Below we will simplify notation by introducing the following spaces.

\begin{defn}\label{H-space}
Let $(M,g)$ be a compact conical manifold and 
assume that the conical cross section $(F,g_F)$ is
strictly tangentially stable. Then we define
\begin{equation*}
\cH^{k, \A}_{\gamma} (M\times [0,T], S) := \hok (M\times [0,T], S_0)_{\gamma}
\ \oplus \ \hok (M\times [0,T], S_1)^b_{\gamma}.
\end{equation*}
If $(F, g_F)$ is tangentially stable but not 
strictly tangentially stable, we set instead
\begin{equation*}
\cH^{k, \A}_{\gamma} (M\times [0,T], S) := \hok (M\times [0,T], S)^b_{\gamma}.
\end{equation*}
The subspaces of time-independent functions are denoted by 
\begin{equation}
\cH^{k, \A}_{\gamma} (M, S) \subset \cH^{k, \A}_{\gamma} (M\times [0,T], S).
\end{equation}
\end{defn}

\section{Mapping properties of the heat operator}\label{mapping-section}

We proceed in the previously set notation on a compact manifold $(M,g)$ with 
isolated conical singularities. Consider the heat equation for the Friedrichs self-adjoint extension $\Delta_L$
of the Lichnerowicz Laplacian  
\begin{equation}\label{heat-equation-formula}
\left(\partial_t + \Delta_L\right) u = 0, \quad u(0) =u_0 \in \dom(\Delta_L).
\end{equation}
Under the assumption of strict tangential stability, the second named 
author constructed in \cite{Ver-Ricci} a fundamental solution $H_L$ to the heat equation
and established the mapping properties for $k \in \N_0$ and $\gamma > 0$ sufficiently small,
cf. \cite[Theorem 3.1 and Theorem 3.3]{Ver-Ricci}
\begin{equation}\label{mapping1}
H_L: \cH^{k, \A}_{-2+\gamma}(M\times [0,T], S)  \to \cH^{k+2, \A}_{\gamma}(M\times [0,T], S).
\end{equation}

\begin{remark}
The H\"older spaces employed in \cite{Ver-Ricci} in fact allow differentiation in space by 
any b-vector field, whereas here we have restricted the admissible differentiation in space
to be given by the covariant derivative. In case of strict tangential stability, this restriction 
is unnecessary, and was introduced here only to treat strict and non-strict tangential stability
cases along each other.
\end{remark}

If strict tangential stability fails and only tangential stability holds, \eqref{mapping1} does 
not hold anymore. However we may prove the following statement. 

\begin{thm}
Let $(M,g)$ be a compact conical manifold and assume that the conical cross section $(F,g_F)$ is
tangentially stable, but not strictly tangentially stable. In this case we assume additionally that
the isolated conical singularity is an orbifold singularity, i.e. $(\cC(F) = (0,1) \times F, dx^2 + x^2 g_F)$
is a flat (not just Ricci-flat) cone. Then for $\gamma, \alpha > 0$ sufficiently small, 
the fundamental solution $H_L$ admits the following mapping property
\begin{equation}\label{mapping2}
H_L: x^{-2+\gamma}\hok(M\times [0,T], S)^b  \to \cH^{k+2, \A}_{\gamma}(M\times [0,T], S).
\end{equation}
\end{thm}

\begin{proof}
If $\cC(F)$ is flat, then $\ker \square_L$ consists of elements that are parallel along 
$F$ and hence vanish under application of $\nabla_{\partial_z}$. This corresponds
precisely to the scalar case, where $\Delta_L$ reduces to the Laplace Beltrami operator
and $\square_L$ is the Laplace Beltrami operator of $(F,g_F)$. In that case, 
$\ker \square_L$ consists of constant functions that vanish under the application of $\partial_z$.
Hence mapping properties in the case of the flat conical singularity can be obtained
along the lines of the estimates for the scalar Laplace Beltrami operator in 
\cite[Theorem 3.3]{Ver-Ricci}.
\end{proof}

In the next result we identify the fundamental solution $H_L$ with the heat operator
of the Friedrichs extension $\Delta_L$ and deduce discreteness of its spectrum.
We assume here that the Lichnerowicz Laplace operator $\Delta_L$ with domain $C^\infty_0(M,S)$
is bounded from below.

\begin{thm}\label{discrete-thm}
Let $(M,g)$ be a compact conical manifold. 
Assume that the Lichnerowicz Laplace operator $\Delta_L$ with domain $C^\infty_0(M,S)$
is bounded from below. Assume moreover that the conical cross section $(F,g_F)$ 
is tangentially stable, and if it is not strictly tangentially stable we assume in 
addition that $\cC(F)$ is an orbifold singularity. Then the following is true.
\begin{enumerate}
\item The Friedrichs self-adjoint extension $\Delta_L$ is bounded from below \\ and the 
fundamental solution $H_L$ equals the heat operator $e^{-t\Delta_L}$.
\item The Friedrichs self-adjoint extension $\Delta_L$ is discrete.
\end{enumerate}
\end{thm}

\begin{proof}
By Proposition \ref{Friedrichs}, the Friedrichs self-adjoint extension $\Delta_L$ is bounded from below
and the heat operator $e^{-t\Delta_L}$ exists by spectral calculus. In order to identify $e^{-t\Delta_L}$
with $H_L$ it suffices to prove that for any fixed $t>0$, $H_L(t)$ maps $L^2(M,S)$ to 
$\dom(\Delta_L)$. This is due to the fact that by definition the heat kernel $e^{-t\Delta_L}$ 
is the unique solution operator to the heat equation \eqref{heat-equation-formula}, which maps $L^2(M,S)$ to 
$\dom(\Delta_L)$. \medskip

Let $(\lambda, \w_\lambda)$ be the set of eigenvalues and 
corresponding eigentensors of the tangential operator $\square_L$.
By the assumption of tangential stability, $\lambda \geq 0$, and we 
consider exactly as in \S \ref{laplace-section}
\begin{equation}
\nu(\lambda) := \sqrt{\lambda + 
\left(\frac{n-1}{2}\right)^2}.
\end{equation}

Consider $M \times M$ with local coordinates $(x,z)$
on the first copy of $M$ near the singularity. 
By construction of $H_L$ in \cite{Ver-Ricci}, the Schwartz kernel\footnote{We denote the Schwartz kernel and the fundamental solution by the same letter.} of the fundamental solution 
$H_L$ defines for a fixed time $t>0$ a polyhomogeneous function on $M \times M$ with 
the asymptotics as $x\to 0$ (we order the eigenvalues $(\lambda)$ of $\square_L$
in the ascending order)
\begin{equation}\label{heat-asymptotics}
H_L(t,x,z,\cdot) \sim \sum_{\lambda} \sum_{j=0}^\infty x^{\nu(\lambda) - \frac{(n-1)}{2}+j} 
\omega_{\lambda}(z) a_{\lambda j} (t, \cdot).
\end{equation}
The coefficients $a_{\lambda j} (t, \cdot) \in C^\infty(M)$ admit for each fixed $t>0$ the same asymptotic expansion as above. 
More precisely, we have as $x\to 0$ 
\begin{equation}\label{heat-asymptotics-coefficients}
a_{\lambda j} (t, x, z) \sim \sum_{\lambda} \sum_{k=0}^\infty 
x^{\nu(\lambda) - \frac{(n-1)}{2}+k} \omega_\lambda(z) b_{\lambda j k} (t),
\end{equation}
where the coefficients $b_{\lambda j k}(t)$ are real numbers.
Consequently, the coefficients $a_{\lambda j}(t, \cdot)$ are in fact elements of $L^2(M,S)$ for any fixed $t>0$. 
Hence we find for any $u \in L^2(M,S)$ as $x\to 0$
\begin{equation}\label{heat-asymptotics-2}
(H_L(t) u)(x,z) \sim \sum_{\lambda} \sum_{j=0}^\infty x^{\nu(\lambda) - \frac{(n-1)}{2}+j} \omega_\lambda(z) \,
\Bigl(a_{\lambda j} (t, \cdot) , u\Bigr)_{L^2}.
\end{equation}
Since the asymptotic expansion \eqref{heat-asymptotics-2} is stable under application of b-vector fields, 
$H_L(t)$ maps $L^2(M,S)$ into $\dom(\Delta_{L,\max})$. We conclude in view of the explicit 
structure of the domain $\dom(\Delta_L)$ in \eqref{Friedrichs-domain} that indeed 
\begin{align}
\forall \, t > 0: \quad H_L(t) : L^2(M,S) \to \dom(\Delta_L).
\end{align}
This proves the first statement. For the second property, note that due to the  
asymptotic expansion above, the Schwartz kernel of $H_L(t)$ is square-integrable on $M \times M$
and hence $H_L(t)$ is Hilbert-Schmidt for any fixed $t>0$. Due to the semi-group property of the heat
operator, $H_L(t) = H_L(t/2) \circ H_L(t/2)$ and hence $H_L(t)$ is trace-class for any fixed $t>0$.
This proves discreteness and the second statement. 
\end{proof}

We conclude the section with a proposition about $\ker \Delta_L$.

\begin{prop}\label{ker-H}
Let $(M,g)$ be a compact conical manifold. 
Assume that the Lichnerowicz Laplace operator $\Delta_L$ with domain $C^\infty_0(M,S)$
is bounded from below. Assume moreover that the conical cross section $(F,g_F)$ 
is tangentially stable. Then for $\gamma, \alpha >0$ sufficiently small such that \eqref{mapping2} holds,
and any $k \in \N_0$ 
$$\ker \Delta_L \subset \cH^{k, \A}_{\gamma} (M, S).$$
\end{prop}

\begin{proof}
First of all, for fixed $t>0$ we employ the asymptotics of the heat
kernel $H_L(t)$ in \eqref{heat-asymptotics} to see that $H_L(t)$ maps 
$L^2(M,S)$ to $\ho(M, S)$ for some $\alpha > 0$. Since $H_L(t) \restriction \ker \Delta_L \equiv \textup{Id}$,
we conclude that $\ker \Delta_L \subset \ho(M, S)$. \medskip

Using again $H_L(t) \restriction \ker \Delta_L \equiv \textup{Id}$ and the mapping properties
\eqref{mapping1} and \eqref{mapping2}, we conclude that $\ker \Delta_L \subset 
\cH^{2, \A}_{\gamma} (M, S) \subset \cH^{2, \A}_{\gamma} (M \times [0,T], S)$. 
Note that for $\gamma>0$ sufficiently small and any $k\in \N_0$ the following inclusions hold by construction 
\begin{equation}
\cH^{k, \A}_{\gamma} \subset 
\cH^{k, \A}_{-2+\gamma} \subset x^{-2+\gamma}\hok(M, S)^b.
\end{equation}
Hence we may apply $H_L(t) \restriction \ker \Delta_L \equiv \textup{Id}$ and the mapping properties
\eqref{mapping1}, \eqref{mapping2} again, to conclude that $\ker \Delta_L \subset 
\cH^{4, \A}_{\gamma} (M, S)$. Iterating the argument, we prove the 
statement $\ker \Delta_L \subset \cH^{k, \A}_{\gamma} (M, S)$
for any $k \in \N_0$.
\end{proof}

\begin{remark}
Note that $g$ lies in the kernel of $\Delta_L$ and nevertheless is only bounded
with respect to itself, without additional weights. This seems to contradict to the
statement that $\ker \Delta_L$ is a subset of a weighted H\"older space
$\cH^{k, \A}_{\gamma} (M\times [0,T], S)$. However, there is no contradiction, once 
we realize that under the decomposition $S=S_0 \oplus S_1$ of symmetric $2$-tensors,
$g = \textbf{1} \cdot g \in C^\infty(M,S_1)$ and hence is trivially an element of   
$\cH^{k, \A}_{\gamma} (M\times [0,T], S)$ for any $k\in \N_0$ and $\gamma > 0$.
\end{remark}

\section{Large time estimates of the heat kernel}\label{large-section}

Let $(M,g)$ be a linearly stable compact conical manifold. Recall that 
by Definition \ref{linear-stability} linear stability is non-negativity of the 
Lichnerowicz Laplace operator $\Delta_L$ with domain $C^\infty_0(M,S)$.
Assume moreover that the conical cross section $(F,g_F)$ 
is tangentially stable, and if it is not strictly tangentially stable we assume in 
addition that $\cC(F)$ is an orbifold singularity. From now on we continue
under this setting unless stated otherwise. \medskip

In this section we are concerned with uniform norm bounds 
for the heat operator $H_L$ as time goes to infinity. Since by Theorem \ref{discrete-thm}
the Friedrichs extension $\Delta_L$ is discrete with non-negative spectrum,
we expect norms of the heat operator restricted to the orthogonal complement of
$\ker \Delta_L$ to decrease exponentially for large times. We prove the following theorem.

\begin{thm}\label{exponential-decay}
Denote the restriction of $H_L$ to the orthogonal complement $\ker \Delta_L^\perp$ by $H_L^\perp$. 
Denote the first non-zero eigenvalue of the Friedrichs extension $\Delta_L$
by $\lambda_1>0$. Fix local generators $\{X_i\}_i$ of $\V_b$ and consider any
$D\in \{\textup{Id}, X_{i_1} \circ \cdots \circ X_{i_\ell} | \ell \in \N\}$. Then for any $t_0>0$ there 
exists a uniform constant $C(t_0)>0$ such that for $t\geq t_0$ and the pointwise norms\footnote{The inner 
product on the symmetric $2$-tensors $S$ is defined with respect to $g$.} 
of the heat kernel $H_L(t)$ taking values in $S \boxtimes S$, and its derivatives
\begin{align}
\| D H^\perp_L(t)(\cdot, \cdot)\|\leq C(t_0) \cdot e^{-t \lambda_1},
\end{align}
where $D$ is applied to the first space variable of $H_L(t)$.
\end{thm}

\begin{proof}
Denote the set of eigenvalues and eigentensors of the Friedrichs extension $\Delta_L$
by $\{\mu, \w_\mu\}$. By discreteness of the spectrum, the heat kernel can be written in terms of 
eigenvalues and eigentensors for any $(p,q) \in M \times M$ by 
\begin{equation}
\begin{split}
&H_L(t)(p,q) = \sum_{\mu\geq 0} e^{-t\mu} \w_\mu(p) \otimes \w_\mu(q), \\
&H^\perp_L(t)(p,q) = \sum_{\mu\geq \lambda_1} e^{-t\mu} \w_\mu(p) \otimes \w_\mu(q).
\end{split}
\end{equation}
Consider any $D\in \{\textup{Id}, X_{i_1} \circ \cdots \circ X_{i_\ell} | \ell \in \N\}$. The notation $(D_1 \circ D_2) H_L$
indicates that the operator $D$ applied once in the first spatial variable of $H_L$ and 
once in the second spatial variable. By the product asymptotics of $H_L$ in \eqref{heat-asymptotics}
and in \eqref{heat-asymptotics-coefficients}, for a fixed $t_0>0$ the pointwise trace 
$\textup{tr}_p H_L(t_0)(p,p)$ admits an asymptotic 
expansion for $p$ approaching the conical singularity, i.e. for $p=(x,z)$ as $x\to 0$
\begin{equation}
\textup{tr}_p H_L(t_0,(x,z)) \sim \sum_{\lambda, \lambda'} \sum_{j,k=0}^\infty x^{\nu(\lambda) + \nu(\lambda')
- (n-1)+i+j} \omega_{\lambda}(z) \omega_{\lambda'}(z) b_{\lambda j k} (t_0).
\end{equation}
This expansion is stable under application of $b$-vector fields and hence for $\lambda_0\geq 0$ being the
smallest eigenvalue of the tangential operator $\square_L$, we conclude 
\begin{equation}
\begin{split}
\textup{tr}_p (D_1 \circ D_2) H_L(t_0,(x,z)) &= O (x^{2\nu(\lambda_0) - (n-1)}) \\ 
&= O(1), \quad \textup{as} \quad x\to 0.
\end{split}
\end{equation}
Hence, $\textup{tr}_p (D_1 \circ D_2) H_L(t_0,p)$ is bounded uniformly in $p\in M$. 
By Proposition \ref{ker-H}, same holds for $H^\perp_L$
and hence there exists $C'(t_0)>0$ such that 
\begin{equation}
\begin{split}
C'(t_0) &\geq \textup{tr}_p (D_1 \circ D_2) H^\perp_L(t_0)(p,p) = 
\sum_{\mu\geq \lambda_1} e^{-t\mu} \| D \w_\mu(p)\|^2 \\
&= e^{-t_0 \lambda_1}\sum_{\mu\geq \lambda_1} e^{-t(\mu-\lambda_1)} \| D \w_\mu(p)\|^2 
=: e^{-t_0 \lambda_1} \cdot K(t_0,p).
\end{split}
\end{equation}
Note that $K(t,p)$ is monotonously decreasing as $t \to \infty$ by construction.
Consequently, for any $t\geq t_0$ and any $p\in M$, we conclude
\begin{equation}
K(t,p) \leq C'(t_0) e^{t_0 \lambda_1} =: C(t_0).
\end{equation}
Hence we can estimate for any $t\geq t_0$ and $p\in M$
\begin{equation}
\begin{split}
\textup{tr}_p (D_1 \circ D_2) H^\perp_L(t)(p,p) = e^{-t \lambda_1} \cdot K(t,p)
\leq C(t_0) e^{-t\lambda_1}.
\end{split}
\end{equation}
The statement is now a consequence of the following estimate
\begin{equation*}
\begin{split}
\| D H^\perp_L(t)(p,q)\| &= 
\sum_{\mu \geq \lambda_1} e^{-t\mu} \| D \w_\mu(p)\|  \cdot \|\w_\mu(q)\|
\\ & \leq \sum_{\mu \geq \lambda_1} \frac{e^{-t\mu}}{2} \| D \w_\mu(p)\|^2 
+ \sum_{\mu \geq \lambda_1}  \frac{e^{-t\mu}}{2} \|\w_\mu(q)\|^2 
\leq C(t_0) e^{-t\lambda_1}.
\end{split}
\end{equation*}
\end{proof}

\begin{cor}\label{exponential-heat-kernel-estimate}
Denote the restriction of $H_L$ to the orthogonal complement $\ker \Delta_L^\perp$ by $H_L^\perp$. 
Denote the first non-zero eigenvalue of the Friedrichs extension $\Delta_L$
by $\lambda_1>0$. Consider the Banach spaces $\cH^{k, \A}_{\gamma} (M\times [0,T], S)$,
defined in Definition \ref{H-space}. Then the heat operators\footnote{without convolution in time} define bounded maps
\begin{equation}\label{H-L-no-time}
\begin{split}
&H_L(t): \cH^{k, \A}_{\gamma} (M, S) \to \cH^{k, \A}_{\gamma} (M, S), \\
&H^\perp_L(t): \cH^{k, \A}_{\gamma} (M, S) \cap \ker \Delta_L^\perp \to \cH^{k, \A}_{\gamma} (M, S),
\end{split}
\end{equation}
bounded uniformly in $t\in (0,T]$ for any fixed $T>0$. Moreover, the operator norm of the latter map 
is bounded by $C \, e^{-t \lambda_1}$ for some constant $C>0$ and all times $t>0$. 
\end{cor}

\begin{proof}
The central results in \cite[Theorems 3.1 and 3.3]{Ver-Ricci} establish mapping properties
\eqref{mapping1} and \eqref{mapping2} for $H_L$ with convolution in time, where time integration 
leads to two additional orders in the front face asymptotics for the various estimates, 
cf. the microlocal heat kernel construction in \cite[\S 2]{Ver-Ricci}. 
However, if we apply $H_L$ to $\cH^{k, \A}_{\gamma} (M, S)$, the additional weights and higher
H\"older regularity offsets the missing time integration and the arguments in \cite{Ver-Ricci} carry over
ad verbatim to the action of $H_L$ without convolution in time. Thus the maps \eqref{H-L-no-time}
are indeed bounded, locally uniformly in $t>0$. \medskip

Theorem \ref{exponential-decay} implies directly that the operator norm of $H^\perp_L(t)$ 
is bounded by $C(t_0) \, e^{-t \lambda_1}$ for $t\geq t_0>0$ and some constant $C(t_0)>0$,
depending on $t_0$. By above, the operator norm of $H^\perp_L(t)$ is bounded uniformly
for $t \in (0,t_0]$. This implies an exponential estimate for all $t>0$ 
with an appropriate constant $C>0$. 
\end{proof}
The mapping properties of the heat kernel have been crucial for establishing shorttime existence for Ricci de Turck flow in the singular setting under the assumption of tangential stability \cite[Theorem 4.1]{Ver-Ricci}. In the orbifold case,
shorttime existence of the Ricci de Turck flow follows from a slight modification of standard parabolic theory \cite{Ham}. Due to uniqueness in the above setting, these different approaches yield the same solution. In particular, smooth inital data produces a smooth solution.
	
The Ricci de Turck flow constructed in \cite{Ver-Ricci} preserves the conical singularity. However, there are also different approaches where Ricci flows have been constructed to smooth out the singularity \cite{De16,ScSi13}. However, the settings in these papers are very different from ours. There, noncompact pure cones are considered as initial metrics whereas in our setting, we start from compact but not nessecarily pure conical metrics. On the other hand, our assumptions on the cross section are more restrictive. Our cross section metric has to be Einstein with the right normalization of the Einstein constant. In \cite{De16,ScSi13}, arbitrary Riemannian metrics with nonnegative resp.\ positive curvature operator are allowed. By combining these assumptions, we would just get quotients of the sphere and the complex projective space as possible cross sections.
\section{Weighted Sobolev spaces}\label{section-weighted}

We proceed in the previously set notation of a compact manifold $(M^m,g)$ with 
an isolated conical singularity with a conical neighborhood $\cC(F)=(0,1) \times F$
over a smooth compact Riemannian manifold $(F^n,g_F)$ with $n=(m-1)$
and a conical metric $g$. As before, $g \restriction \cC(F) = dx^2 \oplus x^2 g_F$
up to higher order terms. In this subsection we define Sobolev and H\"older spaces
on $M$ with values in the vector bundle $E$ associated to $TM$,
and study their embedding and multiplication properties. The vector bundle 
$E$ can be e.g. the vector bundle $S$ of symmetric $2$-tensors, or simply $TM$.

\begin{defn}\label{product-spaces}
Consider $s \in \N_0$ and $\delta \in \R$. Let $x: M \to (0,1]$ be a smooth nowhere
vanishing function which coincides with the radial function over the conical neighborhood
$\cC(F)\subset M$. Let $\nabla$ denote the Levi Civita connection on $E$, induced by $g$, 
and choose local bases $\{X_1, \ldots, X_m\}$ of $\V_b$. We consider the space 
$L^2(M,E)$ of square-integrable sections of $E$ with respect to the volume form of $g$ and
the pointwise inner product $\| \cdot \|_g$ on fibres of $E$ induced by $g$. \medskip

\begin{enumerate}
\item We define the Sobolev space $H^s_\delta (M,E)$ as 
the closure of compactly supported smooth sections $C^\infty_0(M,E)$ under
\begin{align*}
\left\|u\right\|_{H^s_\delta}= \sum_{k=0}^s \sum_{j=1}^m \| x^{k-\delta-\frac{m}{2}} 
\nabla^k_{X_j}u \|_{L^2}.
\end{align*}
Note that $L^2(M,E) = H^0_{-\frac{m}{2}}(M,E)$ by construction. 

\item We define the Banach space $C^{s}_\gamma(M,E)$ 
as the closure of  $C^{\infty}_{0}(M,E)$ under 
\begin{align*}
\left\|u\right\|_{C^s_\delta} = \sum_{k=0}^s \sum_{j=1}^m 
\sup_{q \in M}\| \left(x^{k-\delta} \nabla^k_{X_j} u\right) (q)\|_{g}.
\end{align*}
\end{enumerate}
\end{defn}

The norms for different choices of local bases $\{X_1, \ldots, X_m\}$ of $\V_b$ and different choices
of Riemannian metrics $g$ with isolated conical singularities, are equivalent due to compactness of $M$ and $F$. \medskip

By interpolation and duality, we may define the Sobolev spaces 
$H^s_\delta (M,E)$ for any $s\in \R$. The advantage of these spaces in contrast to the spaces in \S \ref{spaces-section}
is that they satisfy Sobolev embedding properties very similar to the classical 
results. As asserted by Pacini \cite[Corollary 6.8, Remark 6.9]{Pacini} we find the following 
embedding a multiplicative properties.

\begin{thm}\label{embedding-theorem} 
The spaces in Definition \ref{product-spaces} admit the following properties. 
\begin{enumerate}
     \item For $\beta \lvertneqq \delta$ we have 
	$C^{s}_{\delta}(M,E) \subset H^s_{\beta}(M,E)$.
	\item For $N > m/2$ and $\beta \leq \delta$ we have
	$H^{s+N}_{\delta}(M,E) \subset C^{s}_{\beta}(M,E)$.
	\item Consider $s > m/2$. 
	Then the multiplication operation is continuous
	\begin{equation*}
*: \ H^{s}_{\delta_1}(M,E)\times H^{s}_{\delta_2}
(M,E)\to H^{s}_{\delta_1 + \delta_2}(M, E).
	\end{equation*}
	\end{enumerate}
\end{thm}

Consider a second order elliptic differential operator $\Delta$ acting on $u \in C^\infty_0(M, E)$ 
such that for $u$ with compact support in $\cC(F)$
\begin{equation}\label{reg-sing}
\Delta u = \left( - \partial_x^2 - \frac{n}{x} \partial_x + \frac{1}{x^2} \square \right) u + \mathscr{O} \, u,
\end{equation}
where $\mathscr{O}$ is a higher order term, i.e. a second order combination of $b$-vector 
fields $\V_b$, weighted with $x^{-1}$ and functions that are smooth up to $x=0$. 
We can now prove the following auxiliary theorem.

\begin{thm}\label{thm-iso}
Consider an elliptic second order differential operator $\Delta$ acting on $u \in C^\infty_0(M, E)$
with the regular-singular expression \eqref{reg-sing} near the conical singularity. 
Assume that $\Delta u = 0$ admits no solutions in $L^2(M,E)$. Then for a generic $\delta \in (-1,1]$,
excluding the exceptional weights
\begin{equation}
\Lambda := \left\{ \pm \sqrt{\lambda + \left(\frac{n-1}{2} \right)^2} \mid \lambda \in \textup{Spec} \, \square \right\}.
\end{equation}
the following mapping is an isomorphism
\begin{equation}\label{Laplace-vf}
\Delta: H^{2}_{-\frac{m}{2}+1+\delta}(M,E) \to H^{0}_{-\frac{m}{2}-1+\delta}(M,E).
\end{equation}
\end{thm}

\begin{proof}
Classical cone calculus, cf. Mazzeo \cite[Theorem 4.4]{Maz:ETO} asserts that
the mapping \eqref{Laplace-vf} is a Fredholm mapping for $\delta \notin \Lambda$. Hence, perturbing the weight
$\delta>0$, we can always assume that \eqref{Laplace-vf} is Fredholm.

By assumption, any solution to $\Delta u = 0$ is $u \notin L^2(M,E)$.
In particular, since $H^{2}_{-\frac{m}{2}+1+\delta}(M,E) \subset L^2(M,E)$, 
the mapping \eqref{Laplace-vf} has trivial kernel. It remains
to prove triviality of the cokernel. Consider the adjoint to \eqref{Laplace-vf},
defined with respect to the $L^2(M,E)$ pairing,
given by the Laplacian $\Delta$ acting as 
\begin{equation}\label{Laplace-vf-2}
\Delta: H^{0}_{-\frac{m}{2}+1-\delta}(M,E) \to H^{-2}_{-\frac{m}{2}-1-\delta}(M,E).
\end{equation}
Triviality of the cokernel for \eqref{Laplace-vf} is equivalent to triviality of the
kernel for \eqref{Laplace-vf-2}. However, since $H^{0}_{-\frac{m}{2}+1-\delta}(M,E)
\subset L^2(M,E)$, the kernel of \eqref{Laplace-vf-2} is trivial. Hence \eqref{Laplace-vf}
is indeed an isomorphism.
\end{proof}

\section{The space of Ricci-flat metrics with conical singularities}\label{projection-section}

Let $\overline{M}$ be a compact manifold with boundary $\partial M = F$ as before.
We write $\cM$ for the space of Riemannian metrics on the open interior $M$.
Consider a fixed Ricci-flat background metric $h_0$ on $M$ with isolated conical 
singularities in the sense of Definition \ref{cone-metric}. 
We assume that $h_0$ is tangentially stable
and consider for any $k\geq 2$ and $\gamma>0$ the weighted H\"older space 
\begin{align}\label{H-space-2}
\cH := \cH^{k, \A}_{\gamma} (M\times [0,T], S),
\end{align} 
defined in Definition \ref{H-space}.
We assume that the corresponding Lichnerowicz Laplacian $\Delta_{L,h_0}$ with domain $C^\infty_0(M,S)$
is bounded from below and write $\Delta_{L,h_0}$ for the its Friedrichs self-adjoint extension again. 
Same notation holds if $h_0$ is replaced by another Ricci-flat metric $h$ on $M$ with 
a non-negative Lichnerowicz Laplacian $\Delta_{L,h}$.
We can now introduce an integrability condition as follows.

\begin{defn}\label{integrability}
	We say that $h_0$ is \emph{integrable} if for some $\gamma >0$ 
	there exists a smooth finite-dimensional manifold $\cF\subset \cH$ such that 
	\begin{enumerate}
	      \item  $T_{h_0} \cF = \ker \Delta_{L, h_0} \subset \cH$,
		\item all Riemannian metrics $h \in \cF $ are Ricci-flat.
		\end{enumerate}
\end{defn}

\begin{remark}
Due to the structure of $\cH$, any metric $h \in \cF \subset \cH$ is automatically
a metric on $M$ with an isolated conical singularity.  Even more, all metrics $h\in\cF$ in fact admit the same Einstein cross section. Indeed, at first, it follows from the definition of the space $\cH$ that the leading order terms of the metrics $h \in \cH$ 
are pure trace elements with respect to $h_0$ and hence metrics on the cross-section are conformal for all $h\in \cF$. But as these metrics are Einstein as well, they must be the same, see e.g.\ \cite[Theorem 1*]{Kue95}.
\end{remark}

We proceed with some properties of integrable conical manifolds. 

\begin{lem}\label{parallelvf}
	Let $(M,g)$, be a compact manifold with a conical singularity. Then the space of parallel vector fields on $M$ is trivial.
\end{lem}
\begin{proof}
	Let us first consider the case of orbifolds. If $X$ was a parallel vector field on $M$ it would lift in a close neighbourhood of the singularity to a parallel vector field $\tilde{X}$ on a small ball $\tilde{B}$ which is invariant under the action of a discrete group $G$ that fixes the origin $0 \in \tilde{B}$. As $\tilde{X}$ is parallel, it can be extended continuously to the origin where it is invariant under the action of $G$ on $T_0\tilde{B}$ via the tangent map. But as we have an orbifold singularity, this action cannot have invariant subspaces and so $\tilde{X}$ and hence also $X$ must vanish.
\medskip

	In the case where $(M,g)$ has a non-orbifold conical singuarity, and hence is not flat, we recall a theorem by Gallot \cite{Gal79} asserting that a Riemannian cone has irreducible holonomy. Thus, the Riemannian (non-orbifold) cone 
	cannot admit a parallel vector field, and in particular the connection Laplacian on vector fields on any domain $V$ of the cone
	with Neumann boundary conditions, has a trivial kernel. 
	We want to conclude that $(M,g)$ has no parallel vector fields as well. \medskip
	
	If $(M,g)$ admits a parallel vector field, the smallest Neumann eigenvalue for the connection Laplacian on vector fields would be zero for any domain $U\subset M$. However, if we take $U$ very close to the singularity, $(U,g)$ is almost isometric to a small domain $V$ on a non-flat Riemannian cone. Since the eigenvalues of the connection Laplacian acting on vector fields on $V$, 
	with Neumann boundary conditions, are positive by the above, they must also be positive on $U$ as they depend continuously on the metric.
\end{proof}

\begin{prop}\label{smoothvectorbundle}
	Let $(M,h_0)$ be a Ricci-flat metric with a conical singularity and suppose, it is integrable
	with a smooth finite-dimensional manifold $\cF \subset \cH$ of Ricci-flat metrics.
	Then, there exists an open neighbourhood $\cU \subset \cH$ such that for every $h\in \cU\cap\cF$, 
	there exists an injection $i_h:T_h\cF\to\ker(\Delta_{L,h})$.
\end{prop}
\begin{proof} We first consider the case of non-orbifold conical singularities. 
	We will show that the injection is constructed by adding Lie derivatives. First of all, we note that
	for any $l \in \N, l \geq 2$
	\begin{align}\label{intersection0}
	\ker(\Delta_{L,h_0})\cap \left\{\mathcal{L}_X{h_0}\mid X\in H^l_{-m/2+2}(M,TM)\right\}=\left\{0\right\}.
	\end{align}
	This holds for the following reason: Suppose that $X\in H^l_{-m/2+2}(M,TM)$ is such that $\mathcal{L}_X{h_0}\in\ker(\Delta_{L,h_0})$. Then, $0=\Delta_L(\mathcal{L}_X{h_0}) = \mathcal{L}_{\Delta X}{h_0}$, see e.g.\ \cite[pp. 28-29]{Lic61} for the latter equality. Here $\Delta$ denotes the connection Laplace on vector fields. 
	Thus $\Delta X$ is a Killing vector field and, since $h_0$ is Ricci-flat, $\Delta X$ is parallel. 
	Therefore, $\Delta X$ vanishes due to Lemma \ref{parallelvf}. This in turn implies that $\nabla X=0$, 
	where we use integration by parts since $\Delta X \in H^0_{-m/2}(M,TM) = L^2(M, TM)$. 
	Thus, $X$ is a parallel vector field and vanishes due to Lemma \ref{parallelvf}. \medskip
 
     Because $\ker(\Delta_{L,h_0})=T_{h_0}\cF$ and \eqref{intersection0} holds, continuity implies
	\begin{align}\label{intersection}
	T_h\cF\cap \left\{\mathcal{L}_X{h}\mid X\in H^l_{-m/2+2}(M,TM)\right\}=\left\{0\right\}
	\end{align}
	for $h\in \cF$ close enough to $h_0$ in $\cH$.
	Consider now $k\in T_h\cF$. Because $\cF$ consists of Ricci-flat metrics, we find
	\begin{align}\label{Ric-linearization}
	0=\frac{1}{2} \, \Delta_{L,h} k+\mathcal{L}_{\, \mathrm{div} \, k \, - \, \frac{1}{2}\nabla \, \trace k}h,
	\end{align}
	because the right hand side is the linearization of the Ricci tensor at $h$ 
	in the direction of $k$, see e.\ g.\ \cite[p.\ 63]{Besse}.
	Now, write $\tilde{k}=k+\mathcal{L}_Xh$ for some vector field $X$. Then, 
	we compute using \eqref{Ric-linearization}
	\begin{align}\label{Lambda-kappa}
	\Delta_{L,h}\tilde{k}=\Delta_{L,h}k+\mathcal{L}_{\Delta X}h=
	-2\mathcal{L}_{\mathrm{div} \, k \, - \, \frac{1}{2}\nabla \trace k}h+\mathcal{L}_{\Delta X}h.
	\end{align}
Because $k\in\cH$  we conclude $\mathrm{div} k-\frac{1}{2}\nabla \trace k\in 
H^{l-1}_{-1+\varepsilon}(M,TM)\subset H^{l-2}_{-m/2}(M,TM) $ for some $l \geq 2$ and a small constant
  $\varepsilon>0$. \medskip
  
  We now apply Theorem \ref{thm-iso} to the connection Laplacian $\Delta$, where in this case the tangential operator 
  $\square$ is the connection Laplacian acting on sections of the pullback bundle of $TM$ on 
  the cross section $F$. In particular, $\square \geq 0$ and hence for the exceptional weights $\Lambda$ we find
  \begin{equation}
  \Lambda \cap \left(-\frac{m-2}{2}, \frac{m-2}{2}\right) = \varnothing.
  \end{equation} 
Therefore, by Lemma \ref{parallelvf} and Theorem \ref{thm-iso}, where we fix the weight 
$\delta = 1$ which is non-exceptional $\delta \notin \Lambda$ for $m \geq 5$, 
the operator
\begin{align*}
\Delta:H^l_{-m/2+2}(M,TM)\to H^{l-2}_{-m/2}(M,TM)
\end{align*}
is an isomorphism in dimension $m\geq 5$. Therefore, for $m \geq 5$, there exists a unique $X\in H^l_{-m/2+2}(M,TM)$ 
solving the equation $\Delta X=2\mathrm{div} k-\nabla \trace k$ and thus implying 
$\tilde{k}\in \ker(\Delta_{L,h})$ by the formula \eqref{Lambda-kappa}. 
Consequently, the map $i_h:k\mapsto k+\mathcal{L}_Xh$ 
maps $ T_h\cF$ to $ \ker(\Delta_{L,h})$ and it is injective due to \eqref{intersection}. 
\medskip

The case of lower dimension $m \leq 4$ is separate, since in these dimensions
all Ricci-flat conical singularities are orbifold singularities. Indeed, the cross section $F$ 
of a Ricci-flat cone in dimension $m \leq 4$ is an positive Einstein manifold of dimension at
most $3$ and hence a quotient of the sphere. The argument on orbifolds is verbatim as above,
but without weights.
\end{proof}

In the notation below, we also employ the classical Sobolev spaces.

\begin{defn}\label{classical-Sobolev}
Let $(M^m,g)$ be a compact Riemannian manifold of dimension $m$ with an isolated conical singularity.
Let $E$ be a vector bundle over $M$ associated to $TM$, endowed with the Levi-Civita connection of $g$. 
The Sobolev space $H^s (M,E)$ is defined as the closure of compactly supported smooth sections 
$C^\infty_0(M,E)$ under
\begin{align*}
	\left\|u\right\|_{H^s}= \sum_{k=0}^s \sum_{\{i_1, \cdots, i_k\}} \| 
	\left(\nabla_{X_{i_1}} \circ \cdots \circ \nabla_{X_{i_k}} \right) u \|_{L^2},
\end{align*}
where we have chosen a local basis $\{X_1, \ldots, X_m\}$ of $\V_b$. 
\end{defn}

The Sobolev norms for different choices of local bases $\{X_1, \ldots, X_m\}$ of $\V_b$ and different choices
of Riemannian metrics $g$ with isolated conical singularities, are equivalent due to compactness of $M$ and $F$. 
\medskip

Moreover, in contrast to the Sobolev spaces in Definition \ref{product-spaces}, the Sobolev spaces in 
Definition \ref{classical-Sobolev} are not weighted. In case of $s=0$ they are related by 
$L^2(M,E)=H^0(M,E)=H^0_{-m/2}(M,E)$. In case of $E$ being a trivial rank one vector bundle, 
we omit $E$ from the notation and simply write $H^s(M)$.

\begin{lem}[Hardy inequality for manifolds with conical singularities]\label{hardy}
	Let $(M^m,g)$, $m\geq2$ be a compact manifold with a conical singularity\footnote{We do not assume that
	the conical metric $g$ is Ricci-flat.} and let $\rho\in C^{\infty}(M)$, $0<\rho\leq 1$ be a function such that $\rho(q)=d(p,q)$ 
	for all $q$ in a small neighbourhood of the singularity $p$. Then there exists a constant $C>0$ such that
	\begin{align*}
	\int_M u^2\rho^{-2}\dv_g\leq C\left\|u\right\|_{H^1(M)}^2
	\end{align*}
	for all $u\in H^1(M)$. Here, $\dv_g$ denotes the volume element of the metric $g$.
\end{lem}
\begin{proof}It suffices to show the inequality for $u\in C^{\infty}(M)$ compactly supported as the general case follows from an approximation argument.
	Let $R>0$ be so small that $\rho(q)=d(p,q)$ for all $q\in B_{2R}(p)$ and such that $B_{2R}(p)$ is diffeomorphic to $(0,2R)\times F$. Let furthermore $\eta_1\in C^{\infty}(M)$ be a cutoff function such that $\eta_1\equiv 1$ on $B_R(p)$, $\eta_1\equiv 0$ on $M\setminus B_{2R}(p)$, $|\nabla\eta_1|\leq 2/R$ on $M$ and let $\eta_2=1-\eta_1\in C^{\infty}(M)$. Then,
	\begin{align*}
	\int_M u^2\rho^{-2}\dv_g&=\int_{B_{2R}(p)}(\eta_1\cdot u)^2\rho^{-2}\dv_g+\int_{M\setminus B_R(p)}(\eta_2\cdot u)^2\rho^{-2}\dv_g\\&\qquad+2\int_{B_{2R}(p)\setminus B_R(p)}\eta_1\eta_2\cdot u^2\rho^{-2}\dv\\
	&\leq \int_{B_{2R}(p)}(\eta_1\cdot u)^2\rho^{-2}\dv_g+C\left\|u\right\|_{L^2(g)}^2.
	\end{align*}
	By using polar coordinates on $B_{2R}(p)$ and the standard Hardy inequality for functions on $\R$, we get (with $n=\mathrm{dim}(F)$) for the first term
	\begin{align*}
	\int_{B_{2R}(p)}(\eta_1\cdot u)^2\rho^{-2}\dv_g&=\int_F\int_0^{2R}(\eta_1\cdot u)^2x^{n-2}dx\dv_{g_{F,x}}\\
	&\leq \frac{4}{(n-1)^2}\int_F\int_0^{2R}(\partial_x(\eta_1\cdot u))^2x^{n}dx\dv_{g_{F,x}}\\
	&\leq \frac{4}{(n-1)^2}\int_{B_{2R}(p)}|\nabla(\eta_1\cdot u)|^2\dv_g\\
	&\leq \frac{8}{(n-1)^2}\int_{B_{2R}(p)}(|\nabla \eta_1|^2u^2+(\eta_1)^2|\nabla u|^2)\dv_g\\
	&\leq C(n,R)\left\|u\right\|^2_{H^1(g)},
	\end{align*}
	which finishes the proof of the lemma.
\end{proof}
\begin{remark}
	By the Kato inequality $|\nabla|h||^2\leq |\nabla h|^2$, the above inequality also holds for any compactly supported tensor field.
\end{remark}
\begin{thm}\label{strong_stability}
	Let  $(M,h_0)$ be a compact linearly stable Ricci-flat manifold with an isolated conical singularity. Recall that 
	linear stability in the sense of Definition \ref{linear-stability} means that the Lichnerowicz Laplacian 
	$\Delta_L\equiv \Delta_{L,h_0}$ is non-negative. Suppose in addition that the tangential operator 
	acting on the cross section $(F,g_F)$ satisfies the bound
	\begin{align*}
	\Box_L\geq C>-\frac{(n-1)^2}{4}+\frac{1}{4},
	\end{align*}
	where $n$ is the dimension of $F$. Then there exist constants $\epsilon_1,\epsilon_2>0$ such that
	\begin{align*}
	(\Delta_Lk,k)_{L^2}\geq \epsilon_1 \left\|\nabla k \right\|_{L^2}^2+\epsilon_2\left\|k\right\|_{L^2}^2
	\end{align*}
	for all $k\in \ker(\Delta_L)^{\perp}\cap\mathscr{H}$, where $\ker(\Delta_L)^{\perp}$ refers to the $L^2$
	orthogonal complement of the kernel for the Friedrichs extension $\Delta_L$.
\end{thm}
We divide the proof of this theorem into two parts, starting with the following
auxiliary lemma.
\begin{lem}\label{boundedlichnerowicz}
	Let  $(M,h_0)$ be a compact Ricci-flat manifold with an isolated conical singularity and 
	suppose that $h_0$ is linearly stable, i.e. its Lichnerowicz Laplacian $\Delta_L = \nabla^*\nabla-
	2\mathring{R}$ is non-negative. Suppose in addition 
	that there exists $\epsilon>0$ such that $\Delta_L^{\epsilon}=(1-\epsilon)
	\nabla^*\nabla-2\mathring{R}$ (with $\mathring{R}$ defined as in \eqref{defLL}) satisfies
	\begin{align*}
	\inf_{k\in\mathscr{H}}\frac{(\Delta_L^{\epsilon}k,k)_{L^2}}{\left\|k\right\|_{L^2}^2}>-\infty.
	\end{align*}
	Then the assertion of Theorem \ref{strong_stability} holds.
\end{lem}
\begin{proof}
	Let $N= \ker(\Delta_L)^{\perp}\cap\mathscr{H}$ and
	\begin{align*}
	C:=\inf_{k\in N}\frac{(\Delta_Lk,k)_{L^2}}{\left\|k\right\|_{L^2}^2}>0,\qquad D:=\inf_{k\in\mathscr{H}}\frac{(\Delta_L^{\epsilon}k,k)_{L^2}}{\left\|k\right\|_{L^2}^2}>-\infty.
	\end{align*}
	Let $\delta\in[0,1]$ and $\Delta_L^{\epsilon\cdot \delta}=(1-\epsilon\cdot \delta)\nabla^*\nabla-2\mathring{R}=(1-\delta)\Delta_L+\delta\Delta_L^{\epsilon}$. Then we have
	\begin{align*}
	\inf_{k\in N}\frac{(\Delta_L^{\epsilon\cdot \delta}k,k)_{L^2}}{\left\|k\right\|_{L^2}^2}\geq (1-\delta)\inf_{h\in N}\frac{(\Delta_Lk,k)_{L^2}}{\left\|k\right\|_{L^2}^2}+\delta\inf_{k\in N}\frac{(\Delta_L^{\epsilon}k,k)_{L^2}}{\left\|k\right\|_{L^2}^2} =  (1-\delta)\cdot C+\delta\cdot D.
	\end{align*}
	If we assume $\delta<\frac{C}{C-D}$ and set $\epsilon_1=\epsilon\cdot\delta$ and $\epsilon_2=(1-\delta)\cdot C+\delta\cdot D>0$, we get
	\begin{align*}
	(\Delta_Lk,k)_{L^2}-\epsilon_1\left\|\nabla k\right\|_{L^2}^2 = 
	(\Delta_L^{\epsilon\cdot \delta}k,k)_{L^2}\geq \epsilon_2\left\|k\right\|_{L^2}^2,
	\end{align*}
	which finishes the proof of the lemma.
\end{proof}
\begin{prop}\label{boundedcriterion}
	Let $(M,h_0)$ be a compact Ricci-flat manifold with an isolated conical singularity. 
	Suppose that the tangential operator acting on the cross section $(F,g_F)$ satisfies the bound
	\begin{align*}
	\Box_L\geq C>-\frac{(n-1)^2}{4}+\frac{1}{4},
	\end{align*}
	where $n$ is the dimension of $F$. Then, there exists an $\epsilon>0$ such that
	\begin{align*}
	\inf_{k\in\mathscr{H}}\frac{(\Delta_L^{\epsilon}k,k)_{L^2}}{\left\|k\right\|_{L^2}^2}>-\infty.
	\end{align*}	
\end{prop}
\begin{proof}
	Let $U$ be an open neighbourhood of the cone diffeomorphic to $(0,\delta)\times F$ with small $\delta>0$ and choose polar coordinates on $U$ such that $h_0|_U$ can be written as $dx^2+x^2g_F(x)$ where $g_F(x)$ is a family of metrics converging to $g_F$ as $x\to 0$. Choose $U=(0,\delta)\times F$ so small that
	\begin{align*}
	\Box_L^x\geq C>-\frac{(n-1)^2}{4}+\frac{1}{4}
	\end{align*}
	for some constant $C$ and all $x\in (0,\delta)$, where $	\Box_L^x$ is the tangential operator of $g_F(x)$.
	Consider $\Phi: C^\infty_0(U, S) \to C^\infty_0(U,S), \w \mapsto x^{n/2} \w$, which extends to an isometry 
	$\Phi:L^2(U,S; g)\to L^2(U,S; dx^2+g_f(x))$. Then one can show the relation
	\begin{align*}
	\Phi\circ \Delta^{\epsilon}_L\circ \Phi^{-1}=-(1-\epsilon)\partial_x^2+
	\frac{1}{x^2}\left(\Box_L^{x,\epsilon}+\frac{(n-1)^2}{4}-\frac{1}{4}\right),
	\end{align*}
	where $\Box_L^{x,\epsilon}$ is the tangential part of $ \Delta_L^{\epsilon}$ at the metric $g_F(x)$. Provided that $\epsilon>0$ is chosen small enough, we still have the bound
	\begin{align*}
	\Box_L^{x,\epsilon}\geq C>-\frac{(n-1)^2}{4}+\frac{1}{4}.
	\end{align*}
	Therefore we get for any $k \in \cH$
	\begin{align*}
	(\Delta_L^{\epsilon}&k,k)_{L^2}=(1-\epsilon)\left\|\nabla k\right\|_{L^2}^2-2(\mathring{R}k,k)_{L^2}\\
	&=(1-\epsilon)\left\|\nabla k\right\|_{L^2(M\setminus U)}^2-2(\mathring{R}k,k)_{L^2(M\setminus U)}+(1-\epsilon)\left\|\nabla k\right\|_{L^2(U)}^2-2(\mathring{R}k,k)_{L^2( U)}\\
	&\geq -2\sup_{M\setminus U}|R|\cdot\left\|k\right\|^2_{L^2(M\setminus U)}\\&\quad+\int_0^\delta \int_F\left[(1-\epsilon)|\partial_xk|^2+x^{-2}\left\langle\left(\Box_L^{x,\epsilon}+\frac{(n-1)^2}{4}-\frac{1}{4}\right)k,k\right\rangle\right]\dv_{g_F(x)}dx\\
	&\geq-2\sup_{M\setminus U}|R|\cdot\left\|k\right\|^2_{L^2(M)},
	\end{align*}
	which finishes the proof of the proposition.
\end{proof}
Note that Theorem \ref{strong_stability} is a consequence of Lemma \ref{boundedlichnerowicz} and Proposition \ref{boundedcriterion}.
\begin{thm}\label{uniformboundlambda}
	Let $(M,h_0)$ be a Ricci-flat manifold with a conical singularity that is integrable 
	with a smooth finite-dimensional manifold $\cF \subset \cH$ of Ricci-flat metrics, 
	and suppose, the cross-section of the cone is either a spherical space form or strictly tangentially stable. 
	If in addition, $(M,h_0)$ is linearly stable, then
		\begin{enumerate}
		\item for any $h \in \cF$ the corresponding Lichnerowicz Laplacian $\Delta_{L,h}$
		with domain $C^\infty_0(M,S)$ is non-negative\footnote{By Theorem 
		\ref{discrete-thm} its Friedrichs self-adjoint extension is discrete and non-negative.} 
		and spectrum of its Friedrichs self-adjoint 
		extension $\textup{Spec} \, \Delta_{L, h} \backslash \{0\}$ admits a lower bound
		$\lambda_1 > 0$ uniformly (!) in $h \in \cF $,
		\item we have a smooth vector bundle over $\cF $
		\begin{equation}
		\ker := \bigsqcup_{h \in \cF} \ker \Delta_{L, h}.
		\end{equation}
		\end{enumerate}
\end{thm}
\begin{proof}
	In Lemma \ref{smoothvectorbundle}, we have seen that for $h$ close enough to $h_0$ in $\cH$, 
	there exists an injection $i_h:T_h\cF\to \ker(\Delta_{L,h})$ and this injection depends smoothly on $h$ by construction. Therefore, in order to prove the statement, it suffices to show the existence of an $\epsilon>0$ such that $\Delta_{L,h}\geq \epsilon>0$ on $(i_h(T_h\cF))^{\perp}$ if $h$ is close enough to $h_0$. \medskip
	
	At first, we want to remark that all the appearing norms are equivalent for metrics in a $\cH$-neighbourhood. Therefore, we may supress the dependence of the norms in the appearing metrics. Let $\cU$ be a small enough $\cH$-neighbourhood of $h_0$, $h\in \cU\cap\cF$ and let $h_s$ be a curve in $\cU\cap\cF$ joining $h_0$ and $h=h_1$. In the case of a conical singularity with strict tangentially stable cross section, we obtain a uniform constant $C>0$ such that $|R_{h_s}|_{h_s}\leq C\cdot \rho^{-2}$ by the definition of $\cH$. Here, $|R_{h_s}|$ is the norm of the Riemann tensor of $h_s$ measured with respect to $h_s$ and $\rho\in C^{\infty}(M)$ is as in Lemma \ref{hardy} the defining function of the singularity. In the case of an orbifold singularity, we get the same estimate, but the curvature is bounded at the singularity.
	Now by variational formulas of connection and curvature (see e.g. \cite[Lemma A.4]{Kro_Diss}), integration by parts and Lemma \ref{hardy} 
	($h' = \partial_s h_s$)
	\begin{align*}
	\frac{d}{ds}(\Delta_{L,h_s}k,k)_{L^2(h_s)}&=\int_M (\nabla^2h'*k+\nabla h' *\nabla k+h'*\nabla^2k)*k+R*h'*k*k\dv\\
	&=\int_M (\nabla h'*\nabla k*k+h'*\nabla k*\nabla k+ R*h'*k*k)\dv\\
	&\leq C\int_M \rho|\nabla h'|\rho^{-1}|k||\nabla k|+|h'||\nabla k|^2+|h'||k|^2\rho^{-2}) \dv\\
	&\leq C\sup (|h'|+ \rho|\nabla h'|)\left(\int_M |k|^2\rho^{-2}\dv+\left\|\nabla k\right\|_{L^2}^2\right)\\
	&\leq C \left\|h'\right\|_{\cH}\left\|k\right\|_{H^1}^2
	\end{align*}
	for any symmetric $2$-tensor $k\in \cH\subset H^1(M,S)$, where we employ the notation of Definition 
	\ref{classical-Sobolev}. Therefore by integration in $s$,
	\begin{align*}
	|(\Delta_{L,h}k,k)_{L^2(h)}-(\Delta_{L,h_0}k,k)_{L^2(h_0)}|\leq C\left\|h-h_0\right\|_{\cH}\left\|k\right\|_{H^1}^2.
	\end{align*}
	Now let $k\in \ker(\Delta_{L,h_0})^{\perp}\cap\cH=:N_{h_0}$
	By Theorem \ref{strong_stability}, we find 
	\begin{equation}\label{estimate-laplace-epsilon}
	\begin{split}
	(\Delta_{L,h}k,k)_{L^2(h)}&\geq (\Delta_{L,h_0}k,k)_{L^2(h_0)}-C\left\|h-h_0\right\|_{\cH(h_0)}\left\|k\right\|_{H^1(h_0)}^2\\
	&\geq (\epsilon_0-C\left\|h-h_0\right\|_{\cH(h_0)})\left\|k\right\|_{H^1(h_0)}^2
	\geq \frac{\epsilon_0}{2}\left\|k\right\|_{L^2(h_0)}^2,
	\end{split}
	\end{equation}
	for $\left\|h-h_0\right\|_{\cH(h_0)} \leq \frac{\epsilon_0}{2}$.
	It remains to show an analogous estimate for $k\in i_h(T_h\mathscr{F})^{\perp}\cap\cH=:N_{h}$ which is uniform in $h\in \cU\cap\cF$. For this purpose, let $\left\{e_1(h),\ldots e_d(h)\right\}$ be an $L^2(h)$-orthonormal basis of $i_h(T_h\mathscr{F})\subset\ker(\Delta_{L,h})$, chosen in such a way that the inequalities $\left\|e_i(h)-e_i(h_0)\right\|_{\cH}\leq C\left\|h-h_0\right\|_{\cH}$ hold for all $i\in\left\{1,\ldots,d\right\}$. This is possible due to Proposition \ref{smoothvectorbundle}. Let $\Phi_h:N_{h_0}\to N_h$ be the orthogonal projection, given by 
	\begin{align*}
	\Phi: k\mapsto k-\sum_{i=1}^d(k,e_i(h))_{L^2(h)}e_i(h).
	\end{align*}
	Since $\Phi_{h_0} = \textup{Id}$, it is clear by continuity that
	\begin{align*}
	(1-\epsilon_1)\left\|k\right\|_{L^2(h_0)}\leq \left\|\Phi_h(k)\right\|_{L^2(h)}\leq (1+\epsilon_1)\left\|k\right\|_{L^2(h_0)}
	\end{align*}
	for some $\epsilon_1>0$, provided that the neighbourhood $\cU$ is chosen small enough.
	Therefore, $\Phi_h$ is injective map. To obtain surjectivity of $\Phi$, consider the map 
	\begin{align*}
	\Psi_h: \, &N_{h_0} \times i (T_{h_0} \cF) \to \cH = N_h \oplus i (T_h \cF) \\
	&\left(k, \sum_{i=1}^d(k,e_i(h_0))_{L^2(h_0)}e_i(h_0)\right) \mapsto \left(\Phi_h(k), \sum_{i=1}^d(k,e_i(h))_{L^2(h)}e_i(h)\right).
	\end{align*}
	Since $\Psi_{h_0} = \textup{Id}_{\cH}$ and the space of surjective operators is open in the operator
	norm topology, $\Psi_h$ and hence also $\Phi_h$ is surjective for $h$ sufficiently close to $h_0$. Hence 
	$\Phi_h$ is an isomorphism for $h \in \cU \cap \cH$. \medskip
	
	Moreover, we have by definition of $\Phi_h$ and the estimate \eqref{estimate-laplace-epsilon} that 
	\begin{align*}
	(\Delta_{L,h}\Phi_h(k),\Phi_h(k))_{L^2(h)}=(\Delta_{L,h}k,k)_{L^2(h)}
	\geq \frac{\epsilon_0}{4}\left\|\Phi_h(k)\right\|_{L^2(h)}^2
	\end{align*}
	for all $k\in N_{h_0}$ with an $\epsilon_0$ independent of $h$. This yields the desired estimate.
\end{proof}

\begin{remark}
	The second statement of Theorem \ref{uniformboundlambda} still holds if the assumption of linear stability 
	is relaxed to existence of a lower bound for $\Delta_{L,h_0}$. The first statement then changes to a uniform
	lower bound $\lambda_1 > 0$ for the absolute values of the non-zero eigenvalues of $\Delta_{L,h}$ for 
	$h \in \cF$.
\end{remark}

\section{The integrability condition in the flat case}
\begin{prop}
	Let $(M,h_0)$ be a flat manifold with an orbifold singularity. Then it is linearly stable and integrable.
\end{prop}
\begin{proof}
Linear stability of a flat manifold is clear, since in that case $\Delta_L = \nabla^* \nabla$. 
In particular, $k\in \ker(\Delta_L)=\ker(\nabla^* \nabla)$ implies $\nabla k=0$. It therefore just remains to show that $h=k+h_0$ is a flat metric with orbifold singularities (for $k$ small enough), i.e. $\cF =  \ker(\Delta_L)$. At first, there are coordinates around each smooth point with respect to which $(h_0)_{ij}=\delta_{ij}$ so that also the functions $h_{ij}$ are constant in that chart as $h-h_0$ is parallel. Thus, $h$ is also flat. As $h_0$ has orbifold singularities, it is around each singular point isometric to $(B_{\epsilon}/\Gamma,\delta)$ where $\Gamma\subset O(m=n+1)$ is a finite subgroup acting strictly discontinuously on $\mathbb{S}^{n}$. If $k$ is a small parallel tensor on this space, $\delta+k$ lifts to a flat $\Gamma$-invariant metric on $B_{\epsilon}$ which in turn implies that $(B_{\epsilon}/\Gamma,\delta+k)$ is flat and admits an orbifold singularity at $0$. This finishes the proof of the proposition. 
\end{proof}

\section{The integrability condition in the K\"{a}hler case}

To show integrability in the K\"{a}hler case, we now adopt the usual strategy of the compact case. Integrability was recently also shown by Alix Deruelle and the first author \cite{DeKr17} for noncompact K\"{a}hler Ricci-flat manifolds which are asymptotically locally Euclidean (ALE). In Section 2.3 in this paper, proofs of \cite{Bal06,KoSp60,Tia87} and in \cite{Besse} have been adapted to the ALE case. Here, we do the steps in the same order: First, we show that any infinitesimal complex deformation actually integrates to a family of complex structures (Theorem \ref{cpx-str-int}) for which we need an adaption of the $\partial\bar{\partial}$-Lemma for manifolds with isolated conical singularities (Lemma \ref{partialbarpartiallemma}).
Having this family of complex structures, implicit function arguments are used then to construct a family of K\"{a}hler metrics associated to this family of complex structures (Proposition \ref{kahler-def}) and eventually a family of Ricci-flat K\"{a}hler metrics by adding $\partial\bar{\partial}$ of suitable potential functions (Theorem \ref{CY-ALE-Int}).
\medskip

In all these steps, suitable function spaces have to be used in order to apply elliptic regularity.
For the case of a strictly tangentially stable manifold which conical singularity, we use weighted Sobolev spaces whereas in the case of orbifolds, we can work with standard Sobolev spaces. In the latter case, we can lift every appearing object  close to the orbifold singularity by a finite covering to an object on a manifold where we can use local elliptic regularity. \medskip

Let $(M, h_0, J_0)$ be a Ricci-flat K\"ahler manifold with a conical singularity. 
The tangent bundle ${}^{ib} TM$ as well as the exterior bundle $\Lambda^* \left({}^{ib} T^*M\right)$ admit
a bi-grading with respect to the complex structure $J_0$
\begin{equation}\label{bigrading}
{}^{ib} TM = T^{1,0}_{J_0} M \oplus T^{0,1}_{J_0}M, \quad 
\Lambda^* M := \Lambda^*  \left({}^{ib} T^*M\right) = \bigoplus_{(p,q)} \Lambda^{p,q}_{J_0}M.
\end{equation}
Let $k\in C^{\infty}(M,S_0)$ and $k_H,k_A$ its hermitian and anti-hermitian part, respectively.
The hermitian and anti-hermitian part are preserved by $\Delta_L$. 
 We can define $I\in  C^{\infty}(M,{}^{ib} T^*M\otimes {}^{ib} TM)$ and $\kappa\in  C^{\infty}(M,\Lambda^{1,1}_{J_0}M)$ by
	\begin{align}\label{herm+antiherm}
	h_0(X,IY)=-k_A(X,J_0Y),\qquad \kappa(X,Y)=k_H(J_0X,Y).
	\end{align}
	It is easily seen that $I$ is a symmetric endomorphism satisfying $IJ_0+J_0I=0$ and thus can be viewed as 
	$I\in  C^{\infty}(M, \Lambda^{0,1}_{J_0}M\otimes T^{1,0}_{J_0}M)$.
 We have the relations
$I(\Delta_L(k_A))=\Delta_C(I(k_A))$ and $\kappa(\Delta_L(k_H) )=\Delta_H(\kappa(k_H))$, where $\Delta_C=\bar{\partial}^*\bar{\partial}+\bar{\partial}\bar{\partial}^*$ and $ \Delta_H$ are the complex Laplacian and the Hodge Laplacian acting on $C^{\infty}(M,
\Lambda^{0,1}_{J_0}M\otimes T^{1,0}_{J_0}M)$ and $C^{\infty}(M,\Lambda^{1,1}_{J_0}M)$, respectively. 
For details see \cite{Koi83} and \cite[Chap. $12$]{Besse}.
As a consequence, we get exactly as in \cite{Koi83}:
\begin{thm}[Koiso]
	If $(M,h_0,J_0)$ is a Ricci-flat K\"{a}hler manifold with conical singularities, it is linearly stable.
\end{thm}

The next two results hold for any K\"{a}hler manifold $(M,h,J)$ 
with an isolated conical singularity which is either an orbifold or where the cross 
section of the cone is strictly tangentially stable. Since in these results the complex structure $J$ is fixed, we omit the lower 
index in the bi-grading decompositions in \eqref{bigrading}. Let $l\geq \frac{m}{2}+1$, where $
m=\dim(M)$. We start with a version of the $\partial\bar{\partial}$-Lemma
adapted to manifolds with conical singularities.

\begin{lem}[$\partial\bar{\partial}$-Lemma for manifolds with isolated conical singularities]\label{partialbarpartiallemma}
	Let $(M^m,h,J)$ be a K\"{a}hler manifold with a conical singularity and $\delta\geq-\frac{m}{2}$ be a 
	non-exceptional for the Hodge-Dolbeault operator.
	Let $\alpha\in H^l_{\delta}(M,\Lambda^{p,q}M)$. Suppose that
	\begin{itemize}
		\item $\alpha=\partial\beta$ for some $\beta\in H^{l+1}_{\delta+1}(M,\Lambda^{p-1,q}M)$ and $\bar{\partial}\alpha=0$ or
		\item $\alpha=\bar{\partial}\beta$ for some $\beta\in H^{l+1}_{\delta+1}(M,\Lambda^{p,q-1}M)$ and ${\partial}\alpha=0$.
	\end{itemize}
	Then there exists a form $\gamma\in H^{l+2}_{\delta+2}(M,\Lambda^{p-1,q-1}M)$ such that $\alpha=\partial\bar{\partial}\gamma$. Moreover, we can choose $\gamma$ to satisfy the estimate $\left\|\gamma\right\|_{H^{l+2}_{\delta+2}}\leq C\cdot \left\|\alpha\right\|_{H^{l}_{\delta}}$. for some $C>0$.
	
	If $(M,g,J)$ is a K\"{a}hler manifold with orbifold singularities, an analogous assertion holds for forms in unweighted Sobolev spaces.			
\end{lem}
\begin{proof}
	For the orbifold case, the proof is exactly as in \cite[Lemma 5.50]{Bal06}. 
\medskip
	
	For the other case, we argue as follows: 
	Let $d=\partial$ or $d=\bar{\partial}$ and $\Delta=\Delta_{\partial}=\Delta_{\bar{\partial}}$ be the
	Hodge-Dolbeault operator acting on $\Lambda^* M$. Let $\epsilon>0$ be small and $\delta'\in [-m/2,-m/2+\epsilon)$ be a
	non-exceptional weight for $\Delta$. Consider $\Delta$ as an operator $\Delta:H^{l+2}_{\delta'+2}
	(M,\Lambda^{*}M)\to H^{l}_{\delta'}(M,\Lambda^{*}M)$.
		Because of the assumption on $\delta'$, it is Fredholm and we have the $L^2$-orthogonal decomposition
		\begin{align*}
			H^k_{\delta'}(M,\Lambda^{*}M)=\ker(\Delta)\oplus \Delta(H^{l+2}_{\delta'+2}(M,\Lambda^{*}M)).
		\end{align*}
		We define the Green's operator $G:H^{k}_{\delta'}(M,\Lambda^{*}M)\to H^{l+2}_{\delta'+2}(M,\Lambda^{*}M)$ to be zero on $\ker(\Delta)$ and to be the inverse of $\Delta$ on $\ker(\Delta)^{\perp}$.
	    We also have
		\begin{align*}
			d( H^{l+1}_{\delta'+1}(M,\Lambda^{*}M))\oplus d^{*}(H^{l+1}_{\delta'+1}(M,\Lambda^{*}M))=\Delta(H^{l+2}_{\delta'+2}(M,\Lambda^{*}M)).
		\end{align*}
		and $G$ is self-adjoint and commutes with $d$ and $d^*$.
		One now shows that $\gamma=-G\partial^*\bar{\partial}^*G\alpha$ does the job in both cases. The estimate on $\gamma$ follows from construction. Now if $\delta\geq -\frac{m}{2}$ with $\delta\geq\delta'$, we still get $\gamma\in H^{l+2}_{\delta'+2}$ satisfying the above condition. However, as in \cite[p.\ 185]{Joy00}, we can conclude that the equation $\alpha=\partial\bar{\partial}\gamma$ already implies that $\gamma\in H^{l+2}_{\delta+2}$.
\end{proof}

An infinitesimal complex deformation is an endomorphism $I:TM\to TM$ that anticommutes with $J$ and satisfies $\bar{\partial}I=0$ and $\bar{\partial}^*I=0$. By the relation $IJ+JI=0$, $I$ can be viewed as a section of $\Lambda^{0,1}M\otimes T^{1,0}M$.

\begin{thm}\label{cpx-str-int}  Let $(M^m,h,J)$ be a K\"{a}hler manifold with a non-orbifold conical singularity and vanishing first Chern class. Let $\delta> 0$ be non-exceptional for the Hodge-Dolbeault operator and $I\in H^{l}_{\delta}(M,\Lambda^{0,1}M\otimes T^{1,0}M)$ be an infinitesimal complex deformation. Then there exists a smooth family of complex structures $J(t)$ with $J(0)=J$ such that $J(t)-J\in H^{l}_{\delta}(M,T^*M\otimes TM)$ and $J'(0)=I$.
\medskip
	
	If $(M,h,J)$ is a K\"{a}hler manifold with orbifold singularities, an analogous assertion holds for unweighted Sobolev spaces.		
\end{thm}
\begin{proof}
		The proof follows along the lines of Tian's proof by the power series approach \cite{Tia87}:
	We write $J(t)=J(1-I(t))(1+I(t))^{-1}$, where the family $I(t)\in  H^k_{\delta}(M,\Lambda^{0,1}M \otimes T^{1,0}M)$
	has to solve the equation
	\begin{align*}
	\bar{\partial}I(t)+\frac{1}{2}[I(t),I(t)]=0,
	\end{align*}
	where $[\, . \, , . \, ]$ denotes the Fr\"{o}licher-Nijenhuis bracket. If we write $I(t)$ as a formal power series
	$I(t)=\sum_{k\geq 1} I_kt^k$, the coefficients have to solve the equation
	\begin{align*}
	\bar{\partial} I_N+\frac{1}{2}\sum_{k=1}^{N-1}[I_k,I_{N-k}]=0,
	\end{align*}
	inductively for all $N\geq 2$. 
	Because the first Chern class vanishes, $\Lambda^{m,0}M$ is trivial.
	Therefore, we have a natural identification of the bundles $\Lambda^{0,1}M\otimes T^{1,0}M=\Lambda^{n-1,1}M$ by using the holomorphic volume form and we now think of the $I_k$ as being $(m-1,1)$-forms.
	Initially, we have chosen
	$I_1\in H^k_{\delta}(M,\Lambda^{0,1}M\otimes T^{1,0}M)$, given by $I=2I_1J$. 
	By the multiplication property of weighted Sobolev spaces and by the assumption $\delta>0$, $[I_1,I_1]\in H^{k-1}_{\delta-1}(M,\Lambda^{m-1,2}M)$. Using $\bar{\partial} I_1=0$ and $\bar{\partial}^*I_1=0$, one can now show that $\bar{\partial}[I_1,I_1]=0$ and $[I_1,I_1]$ is $\partial$-exact. The $\partial\bar{\partial}$-lemma now implies the existence of a $\psi\in  H^{k+1}_{\delta+1}(M,\Lambda^{m-2,1}M)$ such that $$\partial\bar{\partial}\psi=-\frac{1}{2}[I_1,I_1],$$ and so, $I_2=\partial \psi\in  H^{k}_{\delta}(M,\Lambda^{m-1,1}M)$ does the job.
	Inductively, we get a solution of the equation
	\begin{align*}
	\partial\bar{\partial}\psi=\frac{1}{2}\sum_{k=1}^{N-1}[I_k,I_{N-k}],
	\end{align*}
	by the $\partial\bar{\partial}$-lemma since the right hand side is $\bar{\partial}$-closed and $\partial$-exact (which in turn is true because $\partial I_k=0$ for $1\leq k\leq N-1$).
	Now we can choose $I_N=\partial\psi\in   H^{k}_{\delta}(M,\Lambda^{m-1,1}M)$.
	Convergence of this series for small $t$ is shown by standard elliptic estimates, c.f.\ also \cite[Theorem 1.14]{DeKr17}. In the orbifold case, the steps are the same but we don't need weighted spaces.
\end{proof}
In the following we mean by $(\delta)$ that the exceptional weight $\delta$ only 
appears in the case of a conical singularity which is not an orbifold. In the orbifold case, we 
work with ordinary Sobolev spaces.\medskip

The proof of Theorem \ref{cpx-str-int} provides an analytic immersion 
\begin{equation*}
\Theta:H^{l}_{(\delta)}(M,\Lambda^{0,1}M\otimes T^{1,0}M)\cap \ker(\Delta)\supset U\to H^{k}_{(\delta)}(M,T^*M\otimes TM)
\end{equation*}
by mapping $I\in U$ to $J(1)$ where $J(t)$ is defined by the power series construction in the above proof. By making $U$ small enough, we can ensure that this power series converges for $t=1$ so that the definition of $\Theta$ makes sense.
The image of this map is a smooth manifold of complex structures which we denote by $\mathcal{J}^k_{(\delta)}$ and whose tangent map at $J$ is just the injection. \medskip

\begin{prop}\label{kahler-def}Let $(M,h_0,J_0)$ and
	$\mathcal{J}^l_{(\delta)}$ be as above and let $\delta>0$ be a non-exceptional weight for the Hodge-Dolbeault operator. 
	Then there exists a $H^{l}_{(\delta)}$- neighbourhood $\mathcal{U}$ of $J_0$ and 
	a smooth map $\Phi:\mathcal{J}^l_{(\delta)}\cap \mathcal{U}\to\mathcal{M}^l_{(\delta)}$ 
	which associates to each $J\in \mathcal{J}^l_{(\delta)}\cap \mathcal{U}$ sufficiently 
	close to $J_0$ a metric $h(J)$ which is $H^{l}_{(\delta)} $-close to $h_0$ and K\"{a}hler 
	with respect to $J$. Moreover, we can choose the map $\Phi$ such that 
	$$d\Phi_{J_0}(I)(X,Y)=\frac{1}{2}(k_0(IX,J_0Y)+k_0(J_0X,IY)).$$
\end{prop}
\begin{proof}
	We adopt the strategy of Kodaira and Spencer \cite[Section 6]{KoSp60}. 
	Let $J_t$ be a family of complex structures in $\mathcal{J}^l_{(\delta)}$. Then we can define a map $\Pi^{1,1}_{t}$, given by $\Pi^{1,1}_{t}\omega(X,Y)=\frac{1}{2}(\omega(X,Y)+\omega(J_tX,J_tY))$, which is the canonical map which projects $2$-forms to $J_t$-hermitian $2$-forms.
	We use it to define $J_t$-hermitian forms $\omega_t$ by
	$\Pi^{1,1}_{t}\omega_0(X,Y)=\frac{1}{2}(\omega_0(X,Y)+\omega_0(J_tX,J_tY))$. 
	Here, 
	Let $\partial_t,\bar{\partial}_t$ the associated Dolbeault operators and $\partial_t^*,\bar{\partial}_t^*$ 
	their formal adjoints with respect to the metric $g_t(X,Y):=\omega_t(X,J_tY)$. We simplify notation in 
	\eqref{bigrading} by setting $\Lambda^{p,q}_tM := \Lambda^{p,q}_{J_t}M$. We define a forth-order, 
	self-adjoint and elliptic linear differential operator $E_t:H^l_{(\delta)}(M,\Lambda^{p,q}_tM)\to 
	H^{l-4}_{(\delta-4)}(M,\Lambda^{p,q}_tM)$ by
	\begin{align*}
	E_t=\partial_t\bar{\partial}_t\bar{\partial}^*_t\partial_t^*+\bar{\partial}^*_t\partial_t^*\partial_t\bar{\partial}_t+\bar{\partial}^*_t\partial_t\partial_t^*\bar{\partial}_t+\partial_t^*\bar{\partial}_t\bar{\partial}^*_t\partial_t
	+\bar{\partial}^*_t\bar{\partial}_t+\partial_t^*\partial_t.
	\end{align*}
	It is straightforward to see that that any $\alpha\in \ker_{L^2}(E_t)$ is bounded. Thus, integration by parts shows that 	$\partial_t\alpha=0$, $\bar{\partial}_t\alpha=0$ and $\bar{\partial}^*_t\partial_t^*\alpha=0$, i.e.\ $d\alpha=0$ and $\bar{\partial}^*_t\partial_t^*\alpha=0$ hold simultaneously. In the non-orbifold case, this follows form the fact that $m\geq 5$ in this case.
	As in \cite[Proposition 7]{KoSp60} (except that we use the appropriate weighted spaces), one now shows that
	\begin{align*}
	\ker(d)\cap H^l_{(\delta')}&(M,\Lambda^{p,q}_tM)=\\
	&\partial_t\bar{\partial}_t(H^{l+2}_{(\delta'+2)}(M,\Lambda^{p-1,q-1}_tM))\oplus \ker_{L^2}(E_t)\cap H^l_{(\delta')}(M,\Lambda^{p,q}_tM)
	\end{align*}
	is an $L^2(g_t)$ orthogonal decomposition and that $\dim\ker_{L^2}(E_t)=\dim\ker_{L^2_{\delta'}}(E_t)$ is constant for small $t$. Here, $\delta'\in [-m/2,-m/2+\epsilon)$ is chosen non-exceptional and $\epsilon>0$ is small. Thus there is a smooth family of $L^2(g_t)$-orthogonal projections $\Pi_{E_t}:L^2(M,\Lambda_t^{p,q}M)\to \ker_{L^2}(E_t)$.
	Now we define
	\begin{align*}
	\tilde{\omega}_t&:=\Pi_{E_t}\omega_t+\partial_t\bar{\partial}_tu_t
	=\Pi_{E_t}\Pi^{1,1}_t\omega_0+\partial_t\bar{\partial}_tu_t,
	\end{align*}
	where $u_t\in H^{l+2}_{(\delta+2)}(M)$ is a smooth family of functions such that $u_0=0$ which will be defined later. By construction, $\bar{\omega}_t:=\Pi_{E_t}\omega_t$ is closed and $H^l_{\delta}$-closed to $\omega_0$. Therefore, $\tilde{\omega}_t$ is also closed and differentiating at $t=0$ yields
	\begin{align*}
	\tilde{\omega}'_0=\Pi_{E_0}\omega'_0+\Pi_{E_0}'\omega_0 +\partial_0\bar{\partial}_0u'_0=\omega'_0+\Pi_{E_0}'\omega_0+ \partial_0\bar{\partial}_0u'_0.
	\end{align*}
	Because $d\tilde{\omega}_t=0$, we have $d\tilde{\omega}'_0=0$ and since $J_0'$ is an infinitesimal complex deformation, $d\omega'_0=0$ 
	which implies that  		 
	$$
	\Pi_{E_0}'\omega_0\in \ker(E_0)^{\perp}\cap\ker(d)\cap H^{l}_{(\delta')}(M,\Lambda^{1,1}_0M)
	=\partial_t\bar{\partial}_t(H^{l+2}_{(\delta'+2)}(M)).
	$$ 
	Let now $ v\in H^{l+2}_{(\delta'+2)}(M) $ so that $\partial_0\bar{\partial}_0v=\Pi_{E_0}'\omega_0.$
	Then, define $u_t\in H^{l+2}_{(\delta'+2)}(M) $ by $$u_t:=tv.$$
	By this choice, $
	\tilde{\omega}_0'=\omega'_0$
	and the assertion for $d\Phi_{J_0}(J'_0)=\tilde{h}_0'$ follows immediately.
	Finally, $\tilde{h}_t(X,Y):=\tilde{\omega}_t(X,J_tY)$ is a Riemannian metric for $t$ small enough and it is K\"{a}hler with respect to $J_t$. Moreover, as in the proof of \cite[Proposition 8.4.4]{Joy00}, the fact that $\Pi_{E_0}'\omega_0\in H^k_{(\delta)}$ implies that $u_t\in H^{l+2}_{(\delta+2)}(M)$. Thus $\tilde{h}_t$ is also $H^l_{(\delta)}$-close to $h_0$.
\end{proof}
\begin{remark}
	Let $J_t$ is a smooth family of complex structures in $\mathcal{J}^l_{(\delta)}\cap \mathcal{U}$ and $h_t=\Phi(J_t)$. Then the construction in the proof above shows that $I=J'_0$ and $k=h'_0$ are related by
	\begin{align*}
	k(JX,Y)=-\frac{1}{2}(h(X,IY)+h(IX,Y)).
	\end{align*}
\end{remark}

\begin{lem}\label{tttensors}
	Let $(M,h)$ be a Ricci-flat manifold with a conical singularity and $k\in\ker(\Delta_L)\cap \mathrm{span}(h)^{\perp}$. Then, $tr (k)=0$ and $div (k)=0$.
\end{lem}
\begin{proof}
	First, recall that $\Delta_L$ preserves the splitting $S=S_0\oplus S_1$ and
	acts as the Laplacian on sections of $S^1$, which consists of the pure trace symmetric two-tensors. Therefore, if $k\in \ker(\Delta_L)$, its pure trace part has to be a multiple of the metric so that $k\in\ker(\Delta_L)\cap \mathrm{span}(h)^{\perp}$ implies $tr (k)=0$.
	\medskip
	
	A straightforward calculations show that $div\circ \Delta_L=\Delta\circ div$, where $\Delta$
	is the connection Laplacian of vector fields. Therefore, $div \, h\in \ker (\Delta)=\left\{0\right\}$ as 
	$M$ does not admit parallel vector fields due to Lemma \ref{parallelvf}.
\end{proof}

\begin{thm}\label{CY-ALE-Int}
	Let $(M,h_0,J_0)$ be a Ricci-flat K\"{a}hler manifold where the cross section is either strictly tangentially stable or a space form. Then for any $k\in\ker(\Delta_L)$, there exists a smooth family of metrics $h(t)\in \cH$ with $h(0)=h_0$ and $h'_0=k$, satisfying $\ric_{h(t)}=0$. Each metric $h(t)$ is K\"{a}hler with respect to some complex structure $J(t)$. In particular, $h_0$ is integrable.
\end{thm}
\begin{proof}
	We proceed similarly as in \cite[Chapter 12]{Besse}. Assume we are given a 
	complex structure $J$ close to $J_0$. Assume we have a $(1,1)$-form $\omega$ (the bi-grading $(1,1)$ is with respect to $J$)
	which is $H^{l}_{(\delta)}$-close to $\omega_0$.
Here, $\delta\in (0,\gamma)$ is a non-exceptional weight of the Hodge-Dolbeault operator and $\gamma$ is the weight for $\cH$, see \eqref{H-space-2}.
	 Then we seek a Ricci-flat metric in the cohomology class $[\omega]\in \mathcal{H}^{1,1}_J(M)$. By the $\partial\bar{\partial}$-lemma, there exists a function $f_{\omega}\in H^{l}_{(\delta)}(M)$, such that $i\partial\bar{\partial}f_{\omega}$ is the Ricci form of $\omega$. If $\bar{\omega}\in [\omega]$ and $\bar{\omega}-\omega\in H^{l}(M,\Lambda^{1,1}_{J}M)$, then there is a $u\in H^{l+2}_{(\delta+2)}(M)$ with $\int_M u\text{ } \omega^n=0$ such that $\bar{\omega}=\omega+i\cdot \partial\bar{\partial}u$. Ricci-flatness of $\bar{\omega}$ is equivalent to the equation
	\begin{align*}
	\partial\bar{\partial}f_{\omega}=\partial\bar{\partial}\log\frac{(\omega+i\partial\bar{\partial}u)^n}{\omega^n}=:\partial\bar{\partial}Cal(\omega,u).
	\end{align*}
	Let $\mathcal{J}^l_{(\delta)}$ be as above and $\Delta_J$ the Dolbeaut Laplacian of $J$ and the metric $h(J)$. Then all the $(L^2_{(\delta)})$-cohomologies $\mathcal{H}^{1,1}_{J,(\delta)}(M)=\ker_{L^2_{(\delta)}}(\Delta_J)\cap L^2_{(\delta)}(M,\Lambda^{1,1}_JM)$ are isomorphic for $J\in\mathcal{J}^l_{(\delta)}$ if we  $\mathcal{J}^l_{(\delta)}$ is small enough:
	We have $\mathcal{H}^2_{(\delta)}(M)=\mathcal{H}^{2,0}_{J,(\delta)}(M)\oplus \mathcal{H}^{1,1}_{J,(\delta)}(M)\oplus \mathcal{H}^{0,2}_{J,(\delta)}(M)$. The left hand side is independent of $J$ and the metric $g(J)$ provided by Proposition \ref{kahler-def}. The spaces on the right hand side are kernels of $J$-dependant elliptic operators
	whose dimension depends upper-semicontinuously on $J$. However the sum of the dimensions is constant
	and so the dimensions must be constant as well. \medskip
	
	Thus, there is a natural projection $pr_{J}:\ker(\Delta_{J_0})\to\ker(\Delta_{J})$ which is an isomorphism.
	Let $H^l_{h_0,(\delta)}(M)$ be the space of $H^l_{(\delta)}$-functions whose integral with respect to $h_0$ vanishes.
	We now want to apply the implicit function theorem to the map
	\begin{align*}
	G: \mathcal{J}^l_{(\delta)}\times \mathscr{H}^{1,1}_{J_0,(\delta)}(M)\times H^{l+2}_{h_0,(\delta+2)}(M)&\to \partial\bar{\partial }(H^{l}_{h_0,(\delta)}(M))\\
	(J,\kappa,u)&\mapsto \partial\bar{\partial}Cal(\omega(J)+pr_J(\kappa),u)-\partial\bar{\partial}f_{\omega(J)+pr_J(\kappa)}
\end{align*}
	where $\omega(J)(X,Y):=h(J)(JX,Y)$ and $h(J)$ is the metric constructed in Proposition \ref{kahler-def}.
	We have $G(J_0,0,0)=0$ and the differential restricted to the third component is just given by $\partial\bar{\partial}\Delta:H^{l+2}_{h_0,(\delta+2)}(M)\to  \partial\bar{\partial}(H^{l}_{h_0,(\delta)}(M))$ \cite[p.\ 328]{Besse}, which is an isomorphism since we restrict to functions with vanishing integral.
	To see this, we argue as follows: The injectivitiy follows from the fact that $\Delta$ is injective on $H^{l+2}_{h_0,(\delta+2)}(M)$. Now let $v\in H^{l}_{h_0,(\delta)}(M)$ and $w=\partial\bar{\partial}\in H^{l-2}_{(\delta-2)}(\Delta^{1,1}M)$ and there exists $\tilde{w}\in H^{l}_{(\delta)}(\Delta^{1,1}M)$ such that $\Delta\tilde{w}=w$ where here, $\Delta$ is the Hodge-Dolbeaut Laplacian. If we apply the $\partial\bar{\partial}$-Lemma, we get $u\in H^{l+2}_{h_0,(\delta+2)}(M)$ such that $\partial\bar{\partial}u=\tilde{w}$ and finally concluding $\partial\bar{\partial}\Delta u=\partial\bar{\partial}v$.
	 Therefore we find a map $\Psi$ defined on a small neighbourhood of $(J_0,0)$ such that $G(J,\kappa, \Psi(J,\kappa))=0$.
	\medskip
	
	Let now $k\in\ker(\Delta_{L,h_0})\cap \mathrm{span}(h_0)^{\perp}$ and let $k=k_H+k_A$ its decomposition into a $J_0$-hermitian and a $J_0$-antihermitian part. We want to show that $k$ is tangent to a family of Ricci-flat metrics.
	By the definition of $\cH$ and by elliptic regularity, $k\in H^l_{(\delta)}(M,S)$ for every $l\in\mathbb{N}$ and for some non-exceptional $\delta>0$. We can define $I\in H^l_{(\delta)}(M,\Lambda^{0,1}M\otimes T^{1,0}M)$ and $\kappa\in H^l_{(\delta)}(M,\Lambda^{1,1}_{J_0})(M)$ as in \eqref{herm+antiherm}.
Because $\Delta_CI=0$, $\bar{\partial} I=0$ and $\bar{\partial}^*I=0$. In addition $\kappa\in \mathcal{H}^{1,1}_{J_0}(M)$.
	Let $J(t)=\Theta(t\cdot I)$ be the family of complex structures tangent to $I$ provided by Theorem \ref{cpx-str-int} and $\tilde{\omega}(t)=\tilde{\Phi}(J(t))$ be the associated family of K\"{a}hler forms constructed in Proposition \ref{herm+antiherm}.
	The map $\Psi$ constructed above defines a family of forms $\omega(t)=\tilde{\omega}(t)+pr_{J(t)}(t\cdot \kappa)+i\partial\bar{\partial} \Psi(\tilde{\omega}(t),t\cdot\kappa)$ and an associated family of Ricci-flat metrics $\tilde{h}(t)(X,Y)=\omega(t)(X,J(t)Y)$. 
	This construction provides a smooth map 
	\begin{align*}\Xi:\ker(\Delta_L)\cap \mathrm{span}(h_0)^{\perp}=\ker(\Delta_L)\cap  H^k_{(\delta)}\supset \mathcal{W}\to\mathscr{H},
	\end{align*}
	whose tangent map is the identity. Therefore, its image forms a smooth finite-dimensional manifold of Ricci-flat metrics which are $H^k_{(\delta)}$-close to $h_0$.
	Moreover, by Sobolev embedding for weighted spaces, these metrics are also $\mathscr{H}$-close to $h_0$ but with a possible different parameter $\gamma$.	
	To extend $\Xi$ to a map
	\begin{align*}
	\widetilde{\Xi}:\ker(\Delta_L)\supset \widetilde{\mathcal{W}}\to \mathscr{H},
	\end{align*}
	we just let elements in $\mathrm{span}(g)$ act by multiplication with a constant. By construction, 
	\begin{align*}\mathrm{im}(\widetilde{\Xi})=:\cF
	\end{align*}
    is a finite-dimensional manifold of Ricci-flat metrics close to $h_0$ such that $T_{h_0}\mathscr{F}=\ker(\Delta_L)$. This finishes the proof of the theorem.
\end{proof}

\section{Long time existence and convergence of the Ricci de Turck flow}\label{long-section}

Consider a compact manifold $(M,h_0)$ with an isolated
conical singularity that satisfies the following three assumptions
\begin{enumerate}
\item[(i)] $(M,h_0)$ is tangentially stable in the sense of Definition \ref{tangential-stability-def}, 
\item[(ii)] $(M,h_0)$ is linearly stable in the sense of Definition \ref{linear-stability}, 
\item[(iii)] $(M,h_0)$ is integrable in the sense of Definition \ref{integrability}.
\end{enumerate}
By definition, $\ker$ is a smooth vector bundle over $\cF \cap \cV$, where 
$\cV \subset \cH$ is a sufficiently small open neighborhood and each 
fibre is equipped with an inner product induced by $L^2(\textup{Sym}^2(T^*M),h)$.
Hence there exists a local orthonormal frame $\{e_1, \ldots, e_d\}$, which can be assumed
to be global by taking $\cV$ sufficiently small. Here, $d$ denotes the rank of the vector
bundle $\ker$. We define the projection 
\begin{equation}
\begin{split}
\Pi_h : (\ker \Delta_{L, h_0})^\perp \cap \cH \to (\ker \Delta_{L, h})^\perp \cap \cH,
\\ g \mapsto g-\sum_{i=1}^d \langle g,e_i(h) \rangle \cdot e_i(h),
\end{split}
\end{equation}
where the inner product $\langle g, e_i(h)\rangle$ is taken with respect to $L^2(\textup{Sym}^2(T^*M),h)$.
We now employ this projection to define
\begin{equation}
\begin{split}
\Phi: &(\cF \cap \cV) \times \left((\ker \Delta_{L, h})^\perp \cap \cH\right) \rightarrow h_0 + \cH, 
\\ &(h,g) \mapsto h + \Pi_h (g).
\end{split}
\end{equation}
The differential of $\Phi$ at $(h_0,0)$ acts as identity between the following spaces
(recall $T_{h_0} \cF = \ker \Delta_{L, h_0}$ by integrability of $h_0$)
\begin{equation*}
\begin{split}
d \Phi_{(h_0,0)}: &T_{h_0} \cF \oplus \left((\ker \Delta_{L, h_0})^\perp \cap \cH\right)
\to \cH = T_{h_0} \cF \oplus \left((\ker \Delta_{L, h_0})^\perp \cap \cH\right), 
\end{split}
\end{equation*}
Therefore, by the implicit function theorem on Banach manifolds, there exist sufficiently
small neighborhoods $\cV', \cV'', \cV''' \subset \cH$, such that 
\begin{equation}
\Phi: (\cF \cap \cV') \times \left((\ker \Delta_{L, h})^\perp \cap \cV''\right) 
\rightarrow h_0 + \cV''' 
\end{equation}
is a diffeomorphism. In particular, we may define
\begin{equation}
\Pi: (h_0 + \cV''')  \to \cF \cap \cV', \quad \Pi := \textup{proj}_1 \circ \Phi^{-1}. 
\end{equation}
By construction $(g- \Pi(g)) \in (\ker \Delta_{L, \Pi(g)})^\perp$.
In order to simplify notation, consider the small open neighborhood 
$\cU = B_R(h_0) \subset h_0 + \cV'$ of $h_0 \in \cF$, so that the projection $\Pi$ acts as
$\Pi: \cU \to \cF$, such that for any $g \in \cU$, $(g-\Pi(g)) \in (\ker \Delta_{L,\Pi(g)})^\perp$.
By the implicit function theorem, $\Pi$ is smooth and maps to a small open 
neighborhood of $h_0$ in $\cF$. In particular, by differentiability of $\Pi$, there
exists a uniform constant $c>0$ such that\footnote{We employ the norms of the 
Banach space $\cH$, unless stated otherwise.}
\begin{equation}\label{auxiliary1}
		\forall \, g \in \cU, \, h \in \cU \cap \cF : \| \Pi (g) - h \| \leq c \|g - h\|.
	\end{equation}

We now turn to the estimates in Corollary \ref{exponential-heat-kernel-estimate}
and their (uniform) dependence on $h \in \cF \cap \cU$. By Theorem \ref{uniformboundlambda}, 
the first non-zero eigenvalue of $\Delta_{L,h}$ is uniformly bounded from below by $\alpha>0$
for $h \in \cF \cap \cU$. The constant $c>0$ in Corollary \ref{exponential-heat-kernel-estimate} 
can similarly be chosen locally uniformly constant in $h \in \cF \cap \cU$, and hence for any fixed
$t_0>0$ and $t\geq t_0$
\begin{equation}\label{auxiliary2}
		\forall \, h \in \cU \cap \cF : \| e^{-\Delta_{L,h}}\| \leq c,
		\quad \| e^{-\Delta_{L,h}} \restriction (\ker \Delta_{L,h})^\perp\| \leq ce^{-\alpha t}.
	\end{equation}
We can now establish the following proposition. 

\begin{prop}\label{first-step-existence}
	For any $N \in \N$, there exists $\epsilon > 0$ sufficiently small and 
	$T = T(\epsilon,N)$ (depending also on $\alpha$ and $C$ from above) such that 
	for any $g\in \cU$ with $h = \Pi(g)\in \cU \cap \cF$ and $\| g - h\| \leq \epsilon$, 
	the Ricci-de-Turck flow starting at $g$, with the background metric $h$, 
	exists for time $T>0$ and
	\begin{equation}\label{estimate-after-T}
		\| g(T) - h \| \leq \frac{\epsilon}{N}.
	\end{equation}
\end{prop}

\begin{proof}
Consider the Ricci-de-Turck-flow $g(t)$ starting at $g$, with the background metric $h$.
Then, as worked out by the second named author in \cite{Ver-Ricci}, $g(t)-h$ is a fixed point of
\begin{equation*}
	\Phi_t \w := e^{-t \Delta_{L,h}} \ast Q_2(\w) + e^{-t \Delta_{L,h}} [g - h],
\end{equation*}
where $\ast$ indicates convolution in time, and $x^2 Q_2(\w)$ is a bounded quadratic
combination of $\w, V\w$ and $V \, V' \w$ for some $V, V' \in \V_b$. Consider
\begin{equation*}
	Z_{\mu,T} := \{\w \in \cH \equiv \cH^{k, \alpha}_{\gamma} (M\times [0,T], S) \mid \|\w\| \leq \mu\}.
\end{equation*}
For any $\w \in Z_{\mu,T}$, the quadratic expression $Q_2(\w)$ may be estimated as
\begin{equation}
\| Q_2(\w) \| \leq \tilde{c} \mu^2, \quad \| Q_2(\w) -Q_2(\w') \| \leq 2 \tilde{c} \mu \Vert \w-\w' \Vert,
\end{equation}
for some $\tilde{c}>0$. By \eqref{auxiliary2} we conclude,
assuming $\|g-h\| \leq \epsilon$, that
\begin{equation}
\begin{split}
	&\Vert \Phi_t (\w) \Vert \leq t \cdot c \cdot \tilde{c} \cdot \mu^2 + c e^{-\alpha t} \epsilon, \\
	&\Vert \Phi_t (\w) - \Phi_t(\w') \Vert \leq t \cdot c \cdot \tilde{c} \cdot 2\mu \Vert \w-\w' \Vert.
\end{split}
\end{equation}
The constants $c, \tilde{c}, \alpha > 0$ can be chosen uniformly for all metrics $h \in \cU \cap \cF$.  
Given $N\in \N$ and the constants $c, \tilde{c}>0$, we set for any $\mu > 0$
\begin{align*}
T := \frac{1}{4 \mu c \tilde{c}} \cdot \min \left\{\frac{1}{cN};1 \right\}, \quad
\epsilon := \frac{\mu}{2c}.
\end{align*}
The choice of $T$ and $\epsilon$ yields the estimates
\begin{align*}
T c \tilde{c} \mu^2 \leq  \frac{\mu}{4}, \quad 
c \epsilon \leq  \frac{\mu}{2}, \quad 
2 T c \tilde{c} \cdot \mu \leq \frac{1}{2}.
\end{align*}
Consequently, $\Phi_t$ is a contracting self-map on $Z_{\mu,T}$
and hence by Banach fixed point theorem admits a fixed point $(g(\cdot) - h) \in Z_{\mu,T}$
for any $\mu >0$. Now choose $\mu > 0$ sufficiently small such that
\begin{equation*}
	c e^{-\alpha T} = c \exp \left({- \frac{\alpha}{4 \mu c \tilde{c} 
	\cdot \min (\frac{1}{CN},1)}}\right) < \frac{1}{2N}.
\end{equation*}
Note also that by construction
\begin{equation*}
	T c \tilde{c} \mu^2 = \frac{1}{4 c \tilde{c} \mu} \cdot \min 
	\left\{\frac{1}{cN},1 \right\} \cdot c \tilde{c} \mu^2
	= \frac{\mu}{4} \cdot \min \left\{\frac{1}{cN},1 \right\} \leq \frac{\mu}{4cN} = \frac{\epsilon}{2N}.
\end{equation*}
Hence we may estimate the norm of the fixed point at time $T$
\begin{equation*}
	\Vert g(T)-h \Vert = 
	\Vert \Phi_T(g(t)-h) \Vert \le T \cdot c \tilde{c} \cdot \mu^2 + c e^{-\alpha T} \epsilon\\
	\le \frac{\epsilon}{2N} + \frac{\epsilon}{2N} = \frac{\epsilon}{N}.
\end{equation*}
\end{proof}

Our idea for the proof of long time existence and convergence is now to 
restart the Ricci de Turck flow at $g(T)$ with $\Pi g(T)$ as the background metric.
Thus, along the flow we change the de Turck vector field and the corresponding diffeomorphism 
at each step. 

\begin{cor}
Consider the open neighborhood $\cU = B_R(h_0) \subset \cH$ of $h_0 \in \cF$.
Consider $(N,\epsilon)$ as in Proposition \ref{first-step-existence}, such that additionally, 
$\epsilon < \frac{R}{2}$ and
	\begin{equation}\label{R-bound}
		\frac{\epsilon c}{N} \sum\limits_{j=0}^{\infty} \left(\frac{c+1}{N}\right)^j < \frac{R}{2}.
	\end{equation}
Then for any $g_0\in \cU$ with $h_0 = \Pi(g_0)\in \cU \cap \cF$ and $\| g_0 - h_0\| \leq \epsilon$,
there exists a Ricci de Turck flow, starting at $g_0$, with a change of reference metric at discrete 
times, converging to $g^* \in \cF \cap \cU$ at infinite time.
\end{cor}

\begin{proof}
We consider the Ricci de Turck flow in Proposition \ref{first-step-existence}.
By \eqref{estimate-after-T}
\begin{equation}\label{g-N-1}
\| g(T) - \Pi g_0 \| \leq \frac{\epsilon}{N}.
\end{equation} 
By \eqref{auxiliary1} and the assumption $\| g_0 - h_0\| \leq \epsilon$ we conclude
\begin{equation}\label{g-N-2}
\| \Pi g(T) - \Pi g_0 \| \leq c \|g(T) - \Pi g_0\| \leq \frac{\epsilon c}{N}.
\end{equation}
Consequently, combining \eqref{g-N-1} and \eqref{g-N-2} we conclude
\begin{equation}\label{g-N-3}
\| \Pi g(T) - g(T) \| \leq \| \Pi g(T) - \Pi g_0 \| + \| g(T) - \Pi g_0 \| \leq \frac{\epsilon (c+1)}{N}.
\end{equation}
We can now restart the Ricci de Turck flow at $g(T)$ with $\Pi g(T)$ as the background metric and proceed
iteratively. As long as $g(k T), \Pi g(k T) \in \cU$, for $k \in \N$, we conclude iteratively
(cf. \eqref{g-N-1}, \eqref{g-N-2} and \eqref{g-N-3})
\begin{equation}
\begin{split}
&\Vert g(k T) - \Pi g(k T) \Vert \leq  \frac{ \epsilon (c+1)^k}{N^k}, \\ 
&\Vert g((k+1) T) - \Pi g(k T) \Vert \leq \frac{ \epsilon (c+1)^{k }}{N^{k+1}}, \\
&\Vert \Pi g((k+1) T) - \Pi g(k T) \Vert \leq \frac{ \epsilon c (c+1)^{k }}{N^{k+1}}.
\end{split}
\end{equation}
This leads to the following diagram, where the numbers above arrows indicate the corresponding
distances of metrics, measured in the $\cH$ norm. 

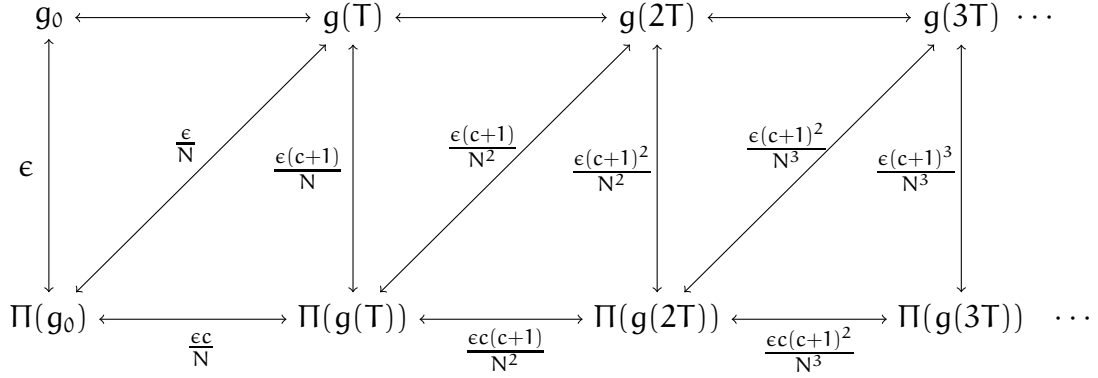
\begin{figure}[h]
\begin{center}
\begin{tikzpicture}

\node (A0) at (0,0) {$\Pi(g_0)$};
\node (A1) at (4,0) {$\Pi(g(T))$};
\node (A2) at (8,0) {$\Pi(g(2T))$};
\node (A3) at (12,0) {$\Pi(g(3T))$};

\node (B0) at (0,4) {$g_0$};
\node (B1) at (4,4) {$g(T)$};
\node (B2) at (8,4) {$g(2T)$};
\node (B3) at (12,4) {$g(3T)$};

\draw [<->] (A0) -- (A1);
\draw [<->] (A1) -- (A2);
\draw [<->] (A2) -- (A3);

\draw [<->] (B0) -- (B1);
\draw [<->] (B1) -- (B2);
\draw [<->] (B2) -- (B3);

\draw [<->] (A0) -- (B0);
\draw [<->] (A1) -- (B1);
\draw [<->] (A2) -- (B2);
\draw [<->] (A3) -- (B3);

\draw [<->] (A0) -- (B1);
\draw [<->] (A1) -- (B2);
\draw [<->] (A2) -- (B3);

\node at (-0.3,2) {$\epsilon$};
\node at (3.4,2) {$\frac{\epsilon (c+1)}{N}$};
\node at (7.4,2) {$\frac{\epsilon (c+1)^2}{N^2}$};
\node at (11.4,2) {$\frac{\epsilon (c+1)^3}{N^3}$};

\node at (1.8,2.3) {$\frac{\epsilon}{N}$};
\node at (5.7,2.3) {$\frac{\epsilon (c+1)}{N^2}$};
\node at (9.7,2.3) {$\frac{\epsilon (c+1)^2}{N^3}$};

\node at (2,-0.4) {$\frac{\epsilon c}{N}$};
\node at (6,-0.4) {$\frac{\epsilon c (c+1)}{N^2}$};
\node at (10,-0.4) {$\frac{\epsilon c (c+1)^2}{N^3}$};

\node at (13,4) {$\cdots$};
\node at (13.5,0) {$\cdots$};

\end{tikzpicture}
\end{center}
\label{edge-picture}
\caption{Iterative sequence of Ricci de Turck flows.}
\end{figure} 

Since for each $k \in \N$ by assumption \eqref{R-bound}
\begin{equation}
\begin{split}
\Vert \Pi g(k T) - h_0 \Vert &\leq \frac{\epsilon c}{N} \cdot 
\sum\limits_{j=0}^{k -1} \frac{(c+1)^j}{N^j} \leq \frac{R}{2}, \\
\Vert g(k T) - h_0 \Vert  &\leq \Vert g(k T) - \Pi g(k T) \Vert + \Vert \Pi g(k T) - h_0 \Vert \\
	& \leq \epsilon \cdot \left(\frac{c+1}{N}\right)^k + \frac{R}{2} \le \epsilon + \frac{R}{2} \le R	,
\end{split}
\end{equation}
$g(k T), \Pi g(k T) \in \cU \equiv B_R(h_0)$ and we can apply Proposition
\ref{first-step-existence} in each step, flowing for another time period $[k T, (k + 1) T]$. 
This proves long-time existence. The sequence $(\Pi g(kT))_{k\in \N}$ is Cauchy, since 
for any $k, \ell \geq n_0$
\begin{equation}
\Vert \Pi g(k T) - \Pi g(\ell T) \Vert \leq \frac{\epsilon c}{N} \cdot 
\sum\limits_{j=n_0}^{\infty} \frac{(c+1)^j}{N^j} \longrightarrow 0,
\end{equation}
as $n_0 \to \infty$. Since $\cF$ is a manifold, the limit $g^* \in \cF \cap \cU$ exists.
We conclude that the sequence $(g(kT))_{k\in \N}$ converges to $g^*$, which 
proves the statement.
\end{proof}

\end{document}